\documentclass[preprint]{imsart}

\usepackage{amsthm,amsmath}

\RequirePackage[numbers]{natbib}
\usepackage{etex}
\usepackage{amsfonts}
\usepackage{float}
\usepackage{morefloats}
\usepackage{subcaption}
\usepackage{tikz}
\usepackage{textcomp}
\usepackage{verbatim}
\usepackage[english]{babel}
\usepackage{pifont}
\usepackage{color}
\usepackage{lscape}
\usepackage{indentfirst}
\usepackage[normalem]{ulem}
\usepackage{booktabs}
\usepackage{latexsym}
\usepackage{url}
\usepackage{epsfig}
\usepackage{graphicx}
\usepackage{amssymb}
\usepackage{bm}
\usepackage{array}
\usepackage[version=4]{mhchem}
\usepackage{ifthen}
\usepackage{multirow}
\usepackage{tkz-euclide}
\usepackage{amstext}
\usepackage{mleftright}
\usepackage{mathtools}
\usepackage{microtype}
\usetkzobj{all}
\usetikzlibrary{arrows.meta,calc,decorations.markings,math,arrows.meta}
\usetikzlibrary{angles}
\usetikzlibrary{backgrounds}
\usetikzlibrary{shapes.misc}
\usetikzlibrary{patterns,quotes}
\DeclareMathOperator\erf{erf}
\DeclareMathOperator\sgn{sgn}

\makeatletter
\def\@ythm#1#2#3[#4]{\def\@currentlabelname{#4}%
  \expandafter\global\expandafter\def\csname#1name\endcsname{#4}%
  \@opargbegintheorem{#3}{\csname the#2\endcsname}{#4}%
  \ifx\thm@starredenv\@undefined
    \thm@thmcaption{#1}{{#3}{\csname the#2\endcsname}{#4}}\fi
  \ignorespaces}

\makeatother

\theoremstyle{plain}
\newtheorem{theorem}{Theorem}[section]

\theoremstyle{definition}

\newtheorem{example}[theorem]{Example}
\newtheorem{prop}[theorem]{Proposition}

\theoremstyle{remark}

\newcommand{\tc}{\text{:}} % tc=tight colon, for use in math mode in Newick trees

\arxiv{}

\begin{document}

\begin{frontmatter}

\title{Hypothesis testing near singularities and boundaries}
\runtitle{Hypothesis testing near singularities and boundaries}
 \author{\fnms{Jonathan D.} \snm{Mitchell},
\corref{}\ead[label=e1]{jdmitchell5@alaska.edu; esallman@alaska.edu; jarhodes2@alaska.edu}}
\author{\fnms{Elizabeth S.} \snm{Allman},}
\and
\author{\fnms{John A.} \snm{Rhodes}}
\address{Department of Mathematics \& Statistics\\ University of Alaska Fairbanks\\ Fairbanks, Alaska 99775, USA\\}

% Address line causes some problems for an unknown reason. If the document fails to compile, then delete \printead{e1}, then compile the document, then put \printead{e1} back in, then compile the document again.

\runauthor{J. D. Mitchell, et al.}

\begin{abstract}
The likelihood ratio statistic, with its asymptotic $\chi^2$ distribution at regular model points, is often used for hypothesis testing. At model singularities and boundaries, however, the asymptotic distribution may not be $\chi^2$, as highlighted by recent work of Drton. Indeed, poor behavior of a $\chi^2$ for testing near singularities and boundaries is apparent in simulations, and can lead to conservative or anti-conservative tests. Here we develop a new distribution designed for use in hypothesis testing near singularities and boundaries, which asymptotically agrees with that of the likelihood ratio statistic. For two example trinomial models, arising in the context of inference of evolutionary trees, we show the new distributions outperform a $\chi^2$.
\end{abstract}

% Say something in abstract about conservative and anti-conservative tests.



\end{frontmatter}

\section{Introduction}

\label{introduction}

The {likelihood ratio statistic} is commonly used to compare a {null model} to an {alternative model}. 
In many circumstances this statistic is asymptotically $\chi^2$-distributed, which greatly facilitates testing of large data sets. 
As is well known, for smaller data sets, or when there are few observations of some outcomes,  a $\chi^2$ 
approximation may not be close enough to the true distribution for reliable testing. Even for large data sets, however, work of \citet{drton2009} has 
highlighted that problems can arise in using a $\chi^2$ approximation at singularities and 
boundaries of the null model. The correct 
asymptotic distribution can be quite different from standard distributions at nearby regular points.

Drton's work shows how one can understand and often calculate an asymptotic distribution at boundaries and singularities, 
but it does not suggest how to use these distributions in practice. Indeed, this is a difficult question, as the 
nature of these asymptotic distributions make clear. For instance, one may find that an asymptotic distribution is 
$\chi_d^2$ with a fixed $d$ degrees of freedom at almost all model points, but that at a boundary or singularity 
it discontinuously jumps to a different distribution 
--- for instance, a mixture of several $\chi^2$ distributions,
or something more complicated.
However, for  the true non-asymptotic distribution, for any fixed sample size no matter how large, we do not expect such a jump to occur.

One might surmise that the asymptotics \emph{at} the singularity or boundary could be relevant to testing even 
{when the true parameter value is} \emph{near} that point, for fixed sample sizes.   As the sample size is increased, the region on which the asymptotics give poor approximations shrinks, but no matter how large a sample is, 
the discontinuous behavior of the asymptotic distribution indicates there is some parameter region on which it is 
inappropriate for empirical use. 
What is needed for a practical test is a different approximation, which is dependent on both sample size and parameter value, 
but still tractable to evaluate.

\medskip

In this work we explore this issue of practical testing near a singularity or boundary, using particular examples of hypothesis testing 
that arise in phylogenomics. Phylogenomics is concerned with inferring evolutionary trees relating several different species from 
genomic-scale data. It builds on phylogenetics (the inference of trees based on sequences of a single gene), but brings in 
population-genetic effects that lead to many inferred gene trees differing from the species (or population) tree. 
Basics of the underlying multispecies coalescent model are explained below, though little familiarity with it is 
necessary for this work. It simply provides two motivating examples of nicely structured and accessible 
submodels of a trinomial ($3$-category multinomial) distribution, for which we can investigate behavior of tests 
near singularities and boundaries. 

Using both theory and simulations, we investigate distributions relevant to hypothesis testing with the 
likelihood ratio statistic for these models.  We illustrate the problematic behavior of a $\chi^2$ distribution near 
boundaries and singularities, even when the sample size is large. We define an alternative 
approximating distribution, and show that it leads to better testing for these models. 
While applications of the material developed here are highly relevant to phylogenomic practice, we defer 
discussion for empiricists to a later paper.

\smallskip

This paper is organized as follows.  In Section~\ref{examples} we lay out basic definitions, and illustrate 
with a simple example the problems that might arise when $\chi^2$ distributions are used to approximate the 
distributions of likelihood ratio statistics near boundaries and 
singularities of null models.  The specifics of the genomic models motivating our primary examples are then introduced. 

The main theorem is given in Section~\ref{mtheorem}, where an approximating distribution is defined for use in 
hypothesis testing. 
In Sections~\ref{sec:T3}~and~\ref{sec:T1}, we specialize to our examples, giving explicit forms of the finite sample approximating 
distributions. By simulation we show that using the standard $\chi_1^2$ for hypothesis testing gives poor performance 
near a boundary or singularity; in contrast, the 
finite sample distributions
we define perform very well for true parameters 
anywhere in the null model.

In Section~\ref{totalvariation}, we use variation distances between the competing distributions ($\chi_1^2$ and ours) to investigate
the region of the null model where the standard $\chi_1^2$ is good for testing, since this depends both on sample size and proximity
to a singularity or boundary point.  The final section is a discussion of our work and its 
potential for application beyond the examples developed here.

\section{Definitions and examples}
\label{examples}

Let $\Theta$, an open subset of $\mathbb{R}^k$, denote the parameter space for a family of probability distributions, 
and $\theta\in\Theta$ an unknown parameter vector.
Submodels are specified by $\Theta_0\subset\widetilde \Theta\subseteq\Theta$, and we formulate the null hypothesis 
$H_0:\theta\in\Theta_0$, with alternative $H_1:\theta\in\Theta_1=\widetilde \Theta\smallsetminus \Theta_0.$
Given some data set, the \emph{likelihood ratio statistic} is
\begin{equation*}
\Lambda=2\left(\sup_{\theta\in\widetilde\Theta}\ell\left(\theta\right)-\sup_{\theta\in\Theta_{0}}\ell\left(\theta\right)\right),
\end{equation*}
where $\ell\left(\theta\right)=\ell\left(\theta\mid \text{data}\right)$ is the {log-likelihood} function.
Then
$\Lambda\ge 0$, and a large value of $\Lambda$ indicates a substantially larger likelihood that $\theta\in \Theta_1$ 
than $\theta\in \Theta_0$, taken as evidence to reject $H_{0}$. By determining the distribution of $\Lambda$ 
under $H_{0}$, the decision as to how large $\Lambda$ must be for rejection can be quantified.

While  it is commonly assumed that the distribution of the likelihood ratio statistic under $H_0$ is well approximated by a $\chi^{2}$ distribution,  establishing this depends on a number of assumptions. \citet{wilks1938large} provided an early justification for sufficiently regular models defined by hyperplanes. \citet{Chernoff54} extended the result to more general models, elucidating the role of the tangent space to the model, and making clear that asymptotic distributions other than $\chi^2$ can arise. Other works emphasize that the statistic may not be asymptotically $\chi^{2}$-distributed at boundaries of $\Theta_0$ (e.g., \citep{miller1977asymptotic}, \citep{self1987asymptotic} and \citep{shapiro1985asymptotic}). 

Recent  research of \citet{drton2009} has emphasized that singularities pose problems as well. 
An asymptotic distribution of the  statistic can be obtained at these problematic model points, as
the distribution of the squared Euclidean distance between a standard normal sample and the appropriately linearly-transformed tangent cone of $\Theta_{0}$ at the true parameter point $\theta_0$ (Theorem~2.6 of \cite{drton2009}). Informally, the {tangent cone} is the set of all possible tangent vectors when approaching $\theta_{0}$ along all possible paths in $\Theta_0$. The tangent cone generalizes the tangent space which lead to the more familiar $\chi^2$ distributions, but may lack the closure properties of a vector space that holds at smooth points of $\Theta_0$.  \medskip

To precisely define singularities and boundaries, we follow \cite{drton2009}. Assume $\Theta_0$ is a \emph{semialgebraic} subset of $\Theta$. That is, $\Theta_0$ is defined by a finite Boolean combination of polynomial equalities and inequalities,
which ensures Chernoff regularity. The Zariski closure, $\overline{\Theta}_{0}$, of $\Theta_0$ is the smallest algebraic variety (the zero set of a finite set of polynomials) containing $\Theta_0$. This closure is the union of
at most finitely many irreducible varieties, called components, which themselves cannot be expressed as a finite union of proper varieties.

A \emph{singularity} of $\Theta_0$ is then either $a)$ a point in $\Theta_0$ which lies on more than one irreducible component of $\overline{\Theta}_{0}$, or $b)$ a point that lies on only one component, but at which the $n\times m$ Jacobian matrix of the defining equations of that component has lower rank than typical. When a point lies on a single irreducible component $\Theta_0^{i}$, the rank of the Jacobian is generically $m-\dim\left(\Theta_0^{i}\right)$. Lower rank indicates a problem with the notion of dimension at the point.

A subset of $\Theta_{0}$ is said to be \emph{open} if it is the intersection of $\overline{\Theta}_{0}$ with an open subset of $\mathbb{R}^{k}$. The \emph{interior} of $\Theta_{0}$ is the union of 
its open subsets, and the \emph{boundary} of $\Theta_0$ is the complement in $\Theta_0$ of its interior. (In most applications, including all our example models, $\Theta_0$ is closed in
$\Theta$ under the standard topology, and this coincides with the usual definition of topological boundary in $\Theta$.)

Note that the boundary and the set of singularities of a model need not be disjoint.

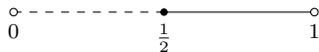
\begin{figure}
\begin{tikzpicture}[xscale=1]
               	\centering
				\draw[-,dashed] (-2,0) -- (0,0);
				\draw[-] (0,0) -- (2,0);
				\draw (-2,0) node[circle,draw=black,fill=white,inner sep=1pt,label=below:{$0$}]{};
				\draw (0,0) node[circle,fill,inner sep=1pt,label=below:{$\frac{1}{2}$}]{};
				\draw (2,0) node[circle,draw=black,fill=white,inner sep=1pt,label=below:{$1$}]{};
				\end{tikzpicture}
     
     \ 
     \caption{ The parameter space for a possibly biased coin. The solid segment is $\Theta_{0}=\left[\frac 12,1\right)$, while  the dashed segment is $\Theta_{1}=\left(0,\frac 12\right)$. } \label{figure1}
\end{figure}

\begin{example}[{\sc Simple model with boundary}]\label{ex:coin}
To test whether a coin, modeled by a Bernoulli random variable with probability of heads $\theta\in\left(0,1\right)$, is 
biased towards tails, formulate hypotheses
\begin{equation*}
H_{0}:\ \theta\geq{}\frac{1}{2},\quad\ \ H_{1}:\ \theta<\frac{1}{2}.
\end{equation*}
Here $\Theta=\left(0,1\right)=\Delta^1$, the open simplex,
and $\Theta_0=\left[\frac{1}{2},1\right)$, as depicted in Figure~\ref{figure1}. 
The Zariski closure of $\Theta_0$ is the real line, and 
$\Theta_0$ has no singularities but a single boundary point $\frac{1}{2}$.
At any $\theta_0$ in the interior of $\Theta_0$ the tangent cone is the full real line, $\left(-\infty,\infty\right)$. However, for $\theta_0 =\frac{1}{2}$
the tangent cone is the half-line $\left[\frac{1}{2},\infty\right)$.

From Theorem~2.6 of \citet{drton2009}, the asymptotic distribution of the likelihood ratio statistic is the distribution of the squared Euclidean distance between a normal random variable centered at $\theta_0$ with variance $1$ and the tangent cone at $\theta_0$.
For $\theta_0>\frac12$, the squared Euclidean distance is $0$ with probability $1$ asymptotically, so the asymptotic distribution is a Dirac delta function $\delta_0$. 
However, for $\theta_0=\frac{1}{2}$ the asymptotic distribution is a mixture $\frac{1}{2}\delta_0+\frac{1}{2}\chi_{1}^{2}$.  Intuitively this is because 
samples from $\mathcal{N}\left(\frac 12,1\right)$ lie on or off the tangent cone $\left[\frac{1}{2},\infty\right)$ with probability $\frac{1}{2}$, and the distributions of the squared distances are $\delta_0$ and $\chi^2_1$ respectively.

For this model, the maximum likelihood estimator 
(MLE) $\hat{\theta}_0$ of the parameter $\theta_0$ is the maximum of $\frac{1}{2}$ and the relative frequency of heads in a sample. If $\theta_{0}$ lies in the 
interior of $\Theta_0$, then for a sufficiently large sample $\hat{\theta}_0$ lies in the interior with probability arbitrarily close to $1$. 
However, for a fixed sample size, no matter how large, there are points $\theta_0$ close to $\frac{1}{2}$ 
but still in the interior of $\Theta_0$ for which this probability is much smaller (in fact, as close to $\frac{1}{2}$ as desired).  
A better approximation to the distribution of the likelihood ratio statistic at such a point might be, for instance, a mixture of $\delta_0$ and the square of a truncated normal centered at $\theta_0$ with variance dependent on sample size. The mixing parameters depend on both $\theta_0$ and the variance, while the truncation point of $\frac{1}{2}$ is not generally the mean of the normal. When $\theta_{0}=\frac{1}{2}$, the normal distribution is truncated at the mean giving the asymptotic mixture distribution already described.

Of course for this model one can simply perform an exact binomial test, without any approximation. Nonetheless, this example highlights 
1) that a likelihood ratio statistic's distribution can fail to converge  uniformly to  a $\chi^{2}$ distribution even on the interior of $\Theta_0$,  
2) the role of the tangent cone in determining correct asymptotics, and 3) the inappropriateness of these asymptotic approximations 
for hypothesis testing for certain parameter values.
\end{example}

\medskip

The next examples are the primary focus of our investigations. We briefly describe their motivation from phylogenomics, with more details
supplied in Appendix~\ref{app:MSC}. The knowledge that these are submodels of a trinomial model is sufficient for the 
remainder of this work.

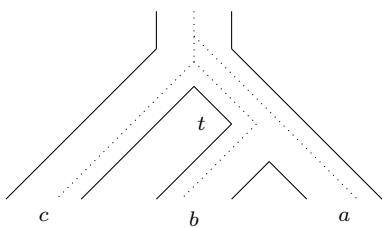
\begin{figure}[!htb]
	\centering
		\begin{tikzpicture}[xscale=1]
        \centering
				\draw[-] (2,0) -- (2,-0.5);
				\draw[-] (3,0) -- (3,-0.5);
				\draw[-] (2,-0.5) -- (0,-2.5);
				\draw[-] (3,-0.5) -- (5,-2.5);
				\draw[-] (2.5,-1) -- (1,-2.5);
				\draw[-] (2.5,-1) -- (3,-1.5);
				\draw[-] (3,-1.5) -- (2,-2.5);
				\draw[-] (3.5,-2) -- (3,-2.5);
				\draw[-] (3.5,-2) -- (4,-2.5);
				\draw[dotted] (2.5,-0.67) -- (0.67,-2.5);
				\draw[dotted] (2.5,0) -- (2.5,-0.67);
				\draw[dotted] (2.5,-0.67) -- (3.33,-1.5);
				\draw[dotted] (3.33,-1.5) -- (2.33,-2.5);
				\draw[dotted] (2.5,-0.33) -- (4.67,-2.5);
                \draw (4.5,-2.5) node[circle,inner sep=1pt,label=below:{$a$}]{};
                \draw (2.5,-2.5) node[circle,inner sep=1pt,label=below:{$b$}]{};
                \draw (0.5,-2.5) node[circle,inner sep=1pt,label=below:{$c$}]{};
                \draw (2.6,-1.25)                node[circle,inner sep=1pt,label=below:{$t$}]{}; \end{tikzpicture}
        \caption{An example of incomplete lineage sorting, where the dotted gene tree topology, $A|BC$, does not match the solid species tree topology, $c|ab$.
        This can occur because gene lineage coalescence events can predate species divergence events, 
        when viewing time backward from the present (upwards).}
        \label{figurecoalescence}
\end{figure}

\begin{example}[{\sc Model T1: Three species related by a specific species tree}]
\label{ex:1tree}
Suppose three species: $a$, $b$ and $c$, are related by a rooted evolutionary species tree as shown in Figure~\ref{figurecoalescence}, 
where the internal branch has length $t\ge 0$.  
Gene trees depicting evolutionary relationships for particular gene lineages ($A$, $B$, $C$) sampled from the three species 
may show differing topological relationships due to the population genetic effect of \emph{incomplete lineage sorting}, 
illustrated in Figure~\ref{figurecoalescence}. 
Under the \emph{multispecies coalescent model} (see Appendix~\ref{app:MSC}), the three possible rooted gene tree topologies  
have probabilities
\begin{equation*}\left(p_{C|AB},\: p_{B|CA}, \:p_{A|BC}\right)=\left(1-\frac 23e^{-t},\: \frac 13 e^{-t},\:\frac 13e^{-t}\right),\end{equation*}
with $C|AB$ denoting the rooted topological gene tree matching the species tree topology with gene lineages $A$ and $B$ 
most closely related, and $B|AC$ and $A|BC$, interpreted analogously, gene tree topologies that do not match that
of the species tree.

For a null hypothesis $H_0$ that the rooted topology of the species tree is $c|ab$, then 
\begin{equation*}
\Theta_0 =\left\{\left(p_1, \; p_2, \; p_3\right) \; \big\vert \; p_1 \geq{} p_2 = p_3 >0, \; \sum_i p_i=1\right \}
 \subset \Delta^2
\end{equation*} 
is shown in Figure~\ref{fig:ex12a}. Here $\Delta^2$ denotes the open 2-dimensional probability simplex. 
The alternative hypothesis, $\Theta_1=\Delta^2\smallsetminus \Theta_0$, can be interpreted either that the species tree has a different 
tree structure $b|ca$ or $a|cb$, or that the model of a simple species tree under the multispecies coalescent is inadequate, 
perhaps due to introgression or hybridization of species populations, population structure within species, or other more complex biological issues.

Samples of $n$ rooted gene trees drawn independently from the multispecies coalescent model on the species tree of Figure~\ref{figurecoalescence} are thus described by a submodel of the trinomial model with parameter space $\Theta_0$.  
\end{example}

\begin{example}[{\sc Model T3: Three species related by any of the three possible species trees}]
\label{ex:3tree}
If the model T1 of Example~\ref{ex:1tree} is modified, so that the specific species tree structure is not fixed, 
but any one of $a|bc$, $b|ac$, or $c|ab$ might be the species tree, then $H_0$ is that there is \emph{some} species tree
giving rise to the gene tree data. 
The alternative $H_1$ is that a simple species tree model does not fit the data.  The null parameter space
$\Theta_0\subset\Delta^2$, shown in Figure~\ref{fig:ex12b}, is the union of three submodels of trinomial models.
\end{example}

As seen in Figure~\ref{fig:ex12a},  the model T1 has a boundary point at $\left(\frac 13,\frac13,\frac13\right) \in \Theta_0$, 
and no singularities. For model T3, the point $\left(\frac 13,\frac13,\frac13\right)$  is a singularity of $\Theta_0$,
since the Zariski closure of $\Theta_0$ is three lines (irreducible components) crossing at that point. 
This point is also a boundary, though we will refer to it simply as the singularity.

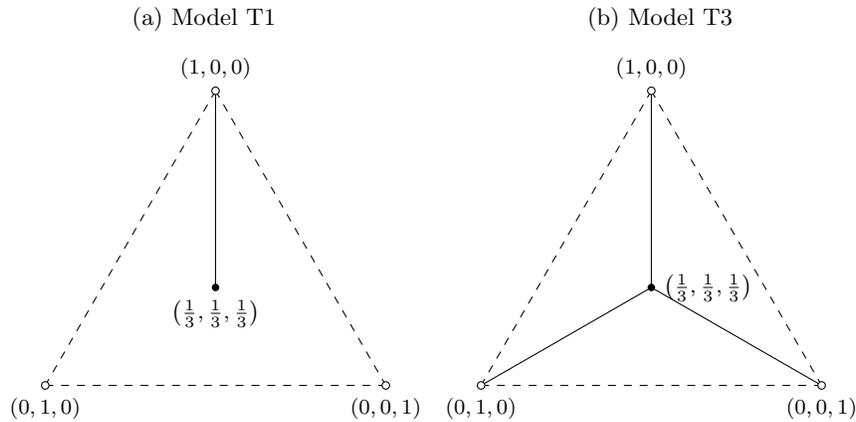
\begin{figure}[!htb]
\begin{subfigure}{0.49\linewidth}
\caption{Model T1}
\label{fig:ex12a}
\hspace*{\fill}
\begin{tikzpicture}[xscale=0.98,yscale=0.98]
				\draw[-] (2.31,0) -- (2.31,-2.67);
				\draw[-,dashed] (2.31,0) -- (0,-4);
				\draw[-,dashed] (2.31,0) -- (4.62,-4);
				\draw[-,dashed] (0,-4) -- (4.62,-4);
				\draw (2.31,0) node[circle,draw=black, fill=white,inner sep=1pt,label=above:{$\left(1,0,0\right)$}]{};
				\draw (2.31,-2.67) node[circle,fill,inner sep=1pt,label=below:{$\left( \frac 13,\frac 13,\frac13\right)$}]{};
                \draw (0,-4) node[circle,draw=black, fill=white,inner sep=1pt,label=below:{$\left(0,1,0 \right)$}]{};
                \draw (4.62,-4) node[circle,draw=black, fill=white,inner sep=1pt,label=below:{$\left(0,0,1\right)$}]{};
\end{tikzpicture}
\end{subfigure}
\hfill
\begin{subfigure}{0.49\linewidth}
\caption{Model T3}
\label{fig:ex12b}
\begin{tikzpicture}[xscale=0.98,yscale=0.98]
				\draw[-] (0,-4) -- (2.31,-2.67);
                \draw[-] (2.31,0) -- (2.31,-2.67);
                \draw[-] (4.62,-4) -- (2.31,-2.67);
                
				\draw[-,dashed] (2.31,0) -- (0,-4);
				\draw[-,dashed] (2.31,0) -- (4.62,-4);
				\draw[-,dashed] (0,-4) -- (4.62,-4);
				\draw (2.31,0) node[circle,draw=black, fill=white,inner sep=1pt,label=above:{$\left(1,0,0\right)$}]{};
				\draw (2.31,-2.67) node[circle,fill,inner sep=1pt,label=right:{$\left( \frac 13,\frac 13,\frac13\right)$}]{};
                \draw (0,-4) node[circle,draw=black, fill=white,inner sep=1pt,label=below:{$\left(0,1,0 \right)$}]{};
                \draw (4.62,-4) node[circle,draw=black, fill=white,inner sep=1pt,label=below:{$\left(0,0,1\right)$}]{};
				\end{tikzpicture}                
\end{subfigure}
\hspace*{\fill}
\caption{Geometric view of the models: (a)  T1 and (b) T3. The solid line segment(s) represent(s) $\Theta_{0}$, while 
the region inside the dotted lines represents $\Theta$, the open probability simplex $\Delta^2$. The central point 
$\left(\frac 13,\frac 13,\frac 13\right)$ corresponding to $t=0$ on any species tree is either a boundary (T1) or a singularity (T3).}
\label{fig:ex12}
\end{figure}

When a rooted species tree on three species has a short internal branch 
so that much incomplete lineage sorting occurs, the expected gene tree probabilities lie close to the 
boundary or singularity $\left(\frac 13,\frac13,\frac 13\right)$ of the models. 
This is exactly the situation in which it is hardest to resolve species tree relationships, and therefore often one of 
pressing biological interest.  Indeed, motivation for this paper is the recognition that the use of the standard asymptotic 
approximation is not reliable near boundaries and singularities, and a careful investigation of this problem is of 
practical as well as theoretical interest.

The models $T_1$ and $T_3$, and the more general multispecies coalescent model for larger trees, 
are increasingly used in inference of species trees from 
genomic-scale data, though
typically little is done to test whether the model is appropriate for data.  For relating three species,  
\citet{degnan2009gene} describe a hypothesis test using a $\chi^2$ distribution, though our work here underscores 
that this test can be problematic near singularities and boundaries. 
Results of \citet{adr2011} show that this test extends to the unrooted $4$-species trees this paper focuses on, 
though the same boundary and singularity issues arise in using the $\chi^2$.
\citet{gaither2016hypothesis} introduce a different hypothesis test for $4$-species trees, but in a different framework, 
working from DNA sequence data under a combined model of coalescence with sequence evolution, and not
on gene tree frequencies. Most empirical studies simply assume the coalescent model on a species tree is 
appropriate, even though several biological processes are known which could violate it.

\medskip

\section{Approximate distributions of likelihood ratio statistics}
\label{mtheorem}

We now illustrate that, in principle, one can obtain an alternative, potentially more useful, approximation to the distribution 
of the likelihood ratio statistic than the asymptotic one.

\medskip

For a statistical model with parameter spaces $\Theta_0\subset\widetilde \Theta\subseteq \Theta$, $\Theta_1=\widetilde \Theta\smallsetminus \Theta_0$,
and parameter $\theta_0\in \Theta_0$ , 
 let $X^{\left(1\right)},\dots, X^{\left(n\right)}$ denote $n$ independent and identically distributed random observations.
 The likelihood function for a sample realization  $X^{\left(1\right)},\dots, X^{\left(n\right)}$ is 
\begin{equation*}\ell_n\left(\theta\right)=\sum_{i=1}^n \log p\left(x^{\left(i\right)}\mid \theta\right).\end{equation*}
Maximizers of the likelihood over $\Theta_0$ and $\widetilde\Theta$ are the maximum likelihood estimators (MLEs) 
over the corresponding parameter spaces.

The likelihood ratio statistic  for a sample then is
\begin{equation*}\Lambda_n=2\left (\sup_{\theta\in \widetilde\Theta} \ell_n\left(\theta\right)-\sup_{\theta\in \Theta_0} \ell_n\left(\theta\right)\right ).\end{equation*}

Under appropriate regularity conditions (see Theorem 16.7 of \citet{van2000asymptotic}) the asymptotic distribution 
of this statistic is that of 
\begin{equation*}
{\left\Vert{}X-\mathcal I\left(\theta_0\right)^{\frac{1}{2}} T_0\right\Vert{}}^2-{\left\Vert{}X-\mathcal I\left(\theta_0\right)^{\frac{1}{2}} T\right\Vert{}}^2,
\end{equation*}
for $X\sim\mathcal N\left(0,I\right)$,
$\mathcal I\left(\theta_0\right)$ the Fisher information matrix at $\theta_0$, $T_0$ and $T$ the tangent cones to $\Theta_0$ and $\widetilde\Theta$ at $\theta_0$, and where $\left\Vert{x - B}\right\Vert$ denotes the minimal Euclidean distance between a point $x$ and set $B$.
In essence, establishing this theorem using local asymptotic normality depends on two approximations: the likelihood ratio process from sample realizations is approximately normal, and the model parameter space is approximated locally by its tangent cone. 

Of these two approximations, it is that of the tangent cone which leads to the discontinuous behavior of the asymptotic distribution, 
since the tangent cone's features behave discontinuously as a function of the parameter. For example, if a model is parameterized 
by a closed ball in  $\mathbb R^k$, at interior points the tangent space will be a $k$-dimensional Euclidean space, while at the 
boundary it becomes a half-space.
For a model with parameter space a curve in the plane that crosses over itself, the tangent space will be a line at most points, 
but at the singularity it is two crossed lines.

Examining a derivation of the asymptotics of the likelihood ratio statistic more closely, local asymptotic normality allows for the approximation by a normal for large samples. For large samples the distribution's covariance approaches {$0$},
and rescaling to a standard normal means the parameter space must be dilated around the true parameter. It is this dilation that allows the parameter space of the model to be approximated by a tangent cone. Thus these two approximations are interrelated, and are not made independently.

Nonetheless, we informally reason that while the normal approximation may be a good one  even for a relatively small sample size, a much larger sample may be needed for the approximating normal to be sufficiently concentrated that the tangent approximation of the model is accurate. This motivates Theorem~\ref{maintheorem} below.  
 
 \medskip
 
For parameter spaces $\Theta_0\subset \widetilde\Theta\subseteq\mathbb R^k$ and parameter value $\theta_0\in \Theta_0$, define 
sequences of scaled translated parameter spaces $T_n=\sqrt n\left(\widetilde\Theta-\theta_0\right)$ and 
$T_{n,0}=\sqrt n \left(\Theta_0 -\theta_0\right)$. Suppose
$T_n\to T$ and $T_{n,0}\to T_0$ in the sense defined in \cite{van2000asymptotic}. As pointed out by \cite{drton2009}, 
a condition such as Chernoff regularity ensures this convergence of spaces, with $T$ and $T_0$ the tangent spaces 
at $\theta_0$ of $\widetilde\Theta$ and $\Theta_0$.
 
\begin{theorem}
\label{maintheorem}
 Consider {$n$ i.i.d. random observations from} a model with parameter space $\Theta$ open in $\mathbb R^k$ and submodels determined by $\Theta_0\subset\widetilde \Theta\subseteq\Theta$, with $\Theta_1=\widetilde \Theta\smallsetminus\Theta_0$.
Let $\theta_{0}\in{}\Theta_0$ be a true parameter point, with non-singular Fisher information matrix $\mathcal I\left(\theta_0\right)$ for a sample of size $1$. 
Let $\mathcal I\left(\theta_{0}\right)^{\frac{1}{2}}$ be a matrix such that
$\mathcal I\left(\theta_{0}\right)=\left(\mathcal I\left(\theta_{0}\right)^{\frac{1}{2}}\right)^T \mathcal I\left(\theta_{0}\right)^{\frac{1}{2}}$ 
and $Y\sim \mathcal{N}\left(\sqrt{n}\mathcal I\left(\theta_{0}\right)^{\frac{1}{2}}\theta_{0},I\right)$.

Then under the {regularity} assumptions of Proposition~16.7 of \cite{van2000asymptotic}, for a sample of size $n$ the likelihood ratio statistic $\Lambda_n$ for $H_0$ vs. $H_1$ is approximately distributed as the random variable
\begin{equation*}W=\inf_{\tau \in  \sqrt n \mathcal I\left(\theta_0\right)^{\frac{1}{2}} \Theta_0} {\left\Vert{}Y-\tau \right\Vert{}}^2-\inf_{\tau \in  \sqrt n \mathcal I\left(\theta_0\right)^{\frac{1}{2}}\widetilde \Theta} {\left\Vert{}Y-\tau\right\Vert{}}^2,\end{equation*}
in the sense that  both the likelihood ratio statistic and this random variable converge in distribution to the same limit as $n\to \infty$.
\end{theorem}
\begin{proof}
By Theorem~16.7 of \cite{van2000asymptotic}, the likelihood ratio statistic converges in distribution to
\begin{equation*}{\left\Vert{}X-\mathcal I\left(\theta_0\right)^{\frac{1}{2}} T_0\right\Vert{}}^2-{\left\Vert{}X-\mathcal I\left(\theta_0\right)^{\frac{1}{2}} T\right\Vert{}}^2,\end{equation*}
for $X\sim\mathcal N\left(0,I\right)$.

However, with $Y = X + \sqrt{n} {\mathcal I \left(\theta_0\right)}^{\frac 1 2} \theta_0$,
\begin{align*}
W=&\inf_{\tau \in  \sqrt n \mathcal I\left(\theta_0\right)^{\frac{1}{2}} \Theta_0} {\left\Vert{} X+\sqrt{n}\mathcal I\left(\theta_{0}\right)^{\frac{1}{2}}\theta_0-\tau \right\Vert{}}^2 \\
&\quad{}-\inf_{\tau \in  \sqrt n \mathcal I\left(\theta_0\right)^{\frac{1}{2}} \widetilde \Theta} {\left\Vert{} X+\sqrt{n}\mathcal I\left(\theta_{0}\right)^{\frac{1}{2}}\theta_0-\tau\right\Vert{}}^2 \\
=&\inf_{\tau \in  T_{n,0}} {\left\Vert{}X-\mathcal I\left(\theta_{0}\right)^{\frac{1}{2}}\tau \right\Vert{}}^2-\inf_{\tau \in  T_n} {\left\Vert{}X-\mathcal I\left(\theta_{0}\right)^{\frac{1}{2}}\tau\right\Vert{}}^2\\
=&{\left\Vert{}X-\mathcal I\left(\theta_{0}\right)^{\frac{1}{2}}T_{n,0} \right\Vert{}}^2-{\left\Vert{}X-\mathcal I\left(\theta_{0}\right)^{\frac{1}{2}}T_n\right\Vert{}}^2.
\end{align*}
Since $T_n\to T$ and $T_{n,0}\to T_0$, applying Lemma~7.13 of \cite{van2000asymptotic} yields the result.
\end{proof}

Note that the condition that the sample is i.i.d. is not necessary in the theorem; a more general result is possible
if $\sqrt{n} {\mathcal I\left(\theta_0\right)}^{\frac 1 2}$ is
replaced with the square root of the Fisher information matrix for a sample of size $n$. 

Moreover, this theorem offers no measure of accuracy of the approximation for any finite sample size, and thus does not 
indicate whether it gives a better approximation than the asymptotic one in practice.   This is typical of results on approximate 
distributions of test statistics.  
To highlight the theorem's potential for improved testing, in subsequent sections we present simulation results indicating
that this distribution outperforms the asymptotic one in our example models T1 and T3. 

\medskip

Though the above theorem is stated for the likelihood ratio statistic, this is but one member of the \emph{power-divergence family} of goodness-of-fit statistics of \citet{cressie1984multinomial}. For multinomially distributed data,  with appropriate assumptions on the null model, all members of the family converge in distribution to the same asymptotic distribution. Thus the theorems and results in this paper are potentially useful for all members of the family. Although
the Neyman-Pearson lemma (\citet{neyman1933problem}) states that the likelihood ratio test is the uniformly most powerful test for simple hypotheses, \citet{cressie1989pearson} highlighted that in other scenarios other family members, such as Pearson's chi-squared statistic, may be better approximated by a $\chi^{2}$ distribution than the likelihood ratio statistic is. It is of interest to investigate the use of the distribution of Theorem~\ref{maintheorem} for these other statistics.

\smallskip

In practice, $\theta_{0}$ and $\mathcal{I}\left(\theta_{0}\right)$ are estimated using the MLE 
$\widehat{\theta}_{0}$. \citet{florescu2014probability} states that a regular exponential family has 
consistent MLEs, and members of a 
regular exponential family satisfy the regularity conditions of \citet{drton2009}. 
However, the approximate distribution of the likelihood ratio statistic in Theorem~\ref{maintheorem} may not be 
accurate for small sample sizes and a consistent $\widehat{\theta}_{0}$ may still be biased for finite samples, in 
which cases attempts should be made to correct the bias.

We emphasize that Theorem~\ref{maintheorem} is focused on obtaining a useful approximate distribution 
near singularities and boundaries of 
$\Theta_0$ within the open parameter space ${\Theta}\subseteq \mathbb R^k$. Close to the topological boundary 
of ${\Theta}$ in $\mathbb R^k$, 
both the approximate distribution of Theorem~\ref{maintheorem} and the $\chi^2$ may perform 
poorly for tests, even with a large sample. This occurs for the models T1 and T3, where ${\Theta}=\Delta^2$, if the 
true parameter is near the vertices of the triangle bounding the simplex. Then frequencies of two tree topologies may be very low, 
and the normal approximation is poor.
When this occurs, other methods such as exact tests or parametric bootstrapping may be used instead. 

\section{Application to Model T3}\label{sec:T3}

We now apply Theorem~\ref{maintheorem} to determine an approximate distribution for the likelihood ratio statistic when 
testing the model T3 vs. an alternative of ``no-species-tree''.
More formally, for $t^{\left(i\right)}$ the branch length in species tree $i\in\left\{1,2,3\right\}$ and taking $\phi_{0}^{\left(i\right)}=e^{-t^{\left(i\right)}} \in \left(0,1\right]$,
the hypotheses are:
\begin{align*}
H_{0}\tc \quad{}\Theta_{0}&=\left\{\left(1-\frac{2}{3}\phi_{0}^{\left(1\right)}, \:\frac{1}{3}\phi_{0}^{\left(1\right)}, \:\frac{1}{3}\phi_{0}^{\left(1\right)}\right)\right\}\cup{}\left\{\left(\frac{1}{3}\phi_{0}^{\left(2\right)}, \:1-\frac{2}{3}\phi_{0}^{\left(2\right)}, \:\frac{1}{3}\phi_{0}^{\left(2\right)}\right)\right\} \\
\noalign{\centering{$\cup{}\left\{\left(\frac{1}{3}\phi_{0}^{\left(3\right)}, \:\frac{1}{3}\phi_{0}^{\left(3\right)}, \:1-\frac{2}{3}\phi_{0}^{\left(3\right)}\right)\right\}$,}}
H_{1}\tc \quad{}\Theta_{1}&=\Delta^2\smallsetminus \Theta_0.
\end{align*}

We view the model $\widetilde{\Theta}=\Theta_0\cup\Theta_1 = \Delta^2$ as a subset of $\mathbb{R}^{2}$ through an appropriate affine transformation (see Appendix~\ref{app1} for full details) which maps the singularity of $\Theta_0$ to the origin and the 
true parameter point $\theta_0 = \left(1-\frac{2}{3}\phi_{0}, \:\frac{1}{3}\phi_{0}, \:\frac{1}{3}\phi_{0}\right)$, without loss of generality, to a point $\left(0,\mu_0\right)$ as in Figure~\ref{transformedSimplex}.  This affine transformation scales the simplex so that the 
normally distributed variable $Y$ of Theorem~\ref{maintheorem} now has mean $\left(0,\mu_0\right)$ and identity covariance, where $\mu_0$ is measured in standard deviations from the singularity and can be interpreted analagously for model T1.
Unless $\theta_{0}=\left(\frac{1}{3},\frac{1}{3},\frac{1}{3}\right)$ the affine transformation 
does not preserves angles.  For other parameter values $\theta_0$, the angle $\alpha_0$ shown in Figure~\ref{transformedSimplex} is less than $\frac{\pi}{6}$.

We make one additional simplification, valid under the assumption that $\theta_0$ is far from the triangle bounding the simplex
$\widetilde{\Theta}$, in a sense dependent on the sample size: the mass of the normal distribution 
of $Y$ outside the image of $\widetilde{\Theta}$ is negligible.   This leads to the following proposition  which is proved 
in Appendix~\ref{app1}.

\begin{prop} \label{prop:T3approx} 
For model T3, the likelihood ratio statistic for testing $H_0$ vs. $H_1$ at a true parameter point $\theta_0= \left(1-\frac{2}{3}\phi_{0}, \:\frac{1}{3}\phi_{0}, \:\frac{1}{3}\phi_{0}\right)$ with sample size $n$ is approximately distributed as the random variable
\begin{equation}\label{eq:T3dis}
\widetilde{\Lambda}_n= \min\left(Z^{2}+\frac{1}{2}\left(1-\sgn\left(\bar{Z}\right)\right)\bar{Z}^{2}, \:\,\sin^2\alpha_0\left (Z+\cot\alpha_0\sgn\left(Z\right)\bar{Z}\right )^2\right),
\end{equation}
where $Z\sim{}\mathcal{N}\left(0,1\right)$, \: $\bar{Z}\sim{}\mathcal{N}\left(\mu_{0},1\right)$, \: $\mu_{0}=\sqrt{2n}\frac{1-\phi_{0}}{\sqrt{\phi_{0}\left(3-2\phi_{0}\right)}}$ and \\
$\alpha_0=\arctan \left( \frac{1}{\sqrt{3\left(3-2\phi_{0}\right)}}\right)$.
\end{prop}

\begin{figure}[!htb]
	\centering
		\begin{tikzpicture}[xscale=1]
        \centering
				\draw[-] (2.51,0) -- (2.51,-2.67);
				\draw[-] (2.51,-2.67) -- (0,-3.58);
				\draw[-] (2.51,-2.67) -- (5.02,-3.58);
				\draw[-,dashed] (2.51,0) -- (0,-3.58);
				\draw[-,dashed] (0,-3.58) -- (5.02,-3.58);
				\draw[-,dashed] (5.02,-3.58) -- (2.51,0);
				\draw[-,dotted] (2.51,-2.67) -- (5.38,-2.67);
				\tkzDefPoint(6.02,-2.67){C}
				\tkzDefPoint(2.51,-2.67){B}
				\tkzDefPoint(5.02,-3.58){A}
				\tkzDefPoint(3.51,-1.97){D}
				\tkzDefPoint(2.51,0){E}
				\tkzMarkAngle[arrows=<->](A,B,C)
				\tkzLabelAngle[pos=0.6](A,B,C){$\alpha_0$}
				\draw (2.51,-1.5) node[circle,fill,inner sep=1.5pt,label=right:{$\left(0,\mu_{0}\right)$}]{};
				\draw (2.51,-2.67) node[circle,fill,inner sep=1.5pt,label=left:{$\ \left(0,{0}\right)$}]{};
				\draw (2.51,0) node[circle,draw=black, fill=white,inner sep=1pt,label=right:{$\left(0,\sqrt{\frac{2n}{\phi_0\left(3-2\phi_0\right)}}\right)$}]{};
				\draw (0,-3.58) node[circle,draw=black, fill=white,inner sep=1pt,label=below:{$\left(-\sqrt{\frac{3n}{2\phi_0}},-\sqrt{\frac{n}{2\phi_0\left(3-2\phi_0\right)}}\right)$}]{};
				\draw (5.02,-3.58) node[circle,draw=black, fill=white,inner sep=1pt,label=below:{$\left(\sqrt{\frac{3n}{2\phi_0}},-\sqrt{\frac{n}{2\phi_0\left(3-2\phi_0\right)}}\right)$}]{};
				\end{tikzpicture}
        \caption{View of the image of model T3 after the affine transformation into $\mathbb R^2$. The singularity is mapped to the origin $\left(0,0\right)$ and the true parameter point $\theta_0$ to $\left(0,\mu_0\right)$. The  mapping is not conformal unless $\theta_0$ is the singularity.}
\end{figure}
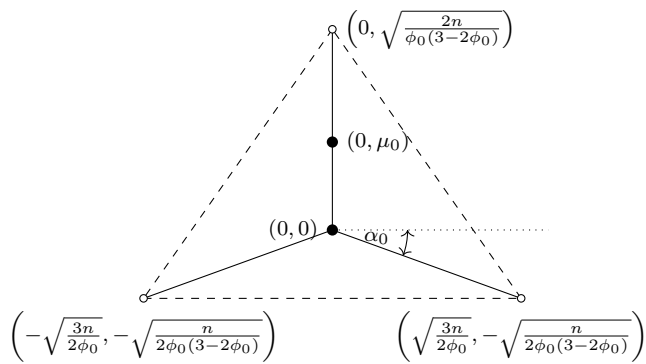\label{transformedSimplex}

Note that all the trigonometric functions in Equation~\eqref{eq:T3dis} can be expressed as algebraic functions of $\phi_0$.

To understand Equation~\eqref{eq:T3dis}, note that $Z$ and $\bar Z$ are random variables corresponding to the $x$ and $y$ components
of the sample point in the transformed space.  The first argument then is simply the squared distance of $\left(Z,\bar Z\right)$ to the vertical half-line
in the null parameter space.  The second argument is the squared distance to the other two half-lines, provided the closest point is
not the origin. $\left(Z,\bar Z\right)$ will be closest to the vertical half-line when the closest point on the other two half-lines is the origin. As shown in the proof, the distance predicted by the first argument of Equation~\eqref{eq:T3dis} is then minimal.
Thus, Equation~\eqref{eq:T3dis} is the minimum squared Euclidean distance between the sample point and the transformed null parameter space.

By replacing $\sgn\left(Z\right)$ and $\sgn\left(\bar Z\right)$ with $\pm 1$, the arguments are easily recognizable as $\chi^2$ 
distributions.  Moreover, suppose $\mu_{0}>0$ corresponds to any non-singular point in $\Theta_{0}$,
then as the sample size $n$ goes to infinity, $\mu_{0}$ also goes to infinity, causing the distribution of 
$\sgn\left(\bar{Z}\right)$ to concentrate on $1$, and the minimum in the formula tends toward selecting the first argument. 
It follows that $\widetilde \Lambda_n$ is asymptotically $\chi_{1}^{2}$-distributed as is the likelihood ratio
statistic $\Lambda_n$, though for $\Lambda_n$ the asymptotic behavior is typically determined more directly 
using the tangent cone approximation.

Now suppose $\mu_{0}=0$, so $\phi_0=1$; that is, the true parameter is the singularity. 
Then for any sample size $n$ the approximate distribution in Equation~\eqref{eq:T3dis} simplifies,
with both $Z$ and $\bar Z$ standard normal.
Although this distribution is not a $\chi^{2}$, it is exactly the asymptotic distribution, found using the 
tangent cone as in \cite{drton2009}. This is not surprising, as
the tangent cone at this point locally agrees with the model itself.
\medskip

Additional computations in Appendix~\ref{app1} give the following. 
\begin{prop} \label{prop:T3pdf}
The probability density function for the random variable $\widetilde \Lambda_n$
given for model T3 in Proposition~\ref{prop:T3approx}  is, for $\lambda > 0$,
\begin{align}
\label{eq:T3pdf}
&f_{\widetilde{\Lambda}_n}\left(\lambda\right)=
\frac{1}{2\sqrt{2\pi{}\lambda}}
\biggl [ 
\exp \left ({-\frac{\lambda}{2}}\right )\left(1-\erf\left(\frac{1}{\sqrt{2}}\left ( \sqrt \lambda\tan  \beta_0 -\mu_{0}\right)\right)\right)
\notag\\&+
\exp \left ({-\frac{1}{2}\left(\sqrt \lambda-\mu_{0}\cos\alpha_{0}\right)^{2}}\right )\left(1-\erf\left(\frac{1}{\sqrt{2}}\left(\sqrt \lambda\tan\beta_0 +\mu_{0}\sin\alpha_{0}\right)\right)\right)
\notag\\&+
\exp\left ( {-\frac{1}{2}\left(\sqrt \lambda+\mu_{0}\cos\alpha_{0}\right)^{2}}\right )\left(1-\erf\left(\frac{1}{\sqrt{2}}\left(\sqrt \lambda\tan\alpha_{0}+\mu_{0}\sin\alpha_{0}\right)\right)\right)
\biggr ],
\end{align}
where $\mu_{0}=\sqrt{2n}\frac{1-\phi_{0}}{\sqrt{\phi_{0}\left(3-2\phi_{0}\right)}}$, $\alpha_{0}=\arctan \frac 1{\sqrt{3\left(3-2\phi_0\right)}}$ and $\beta_0=\frac{1}{2}\left(\frac \pi 2 -\alpha_{0}\right)$.
\end{prop}

\begin{figure}[!htb]
\hspace*{\fill}
\includegraphics[width=.8\linewidth]{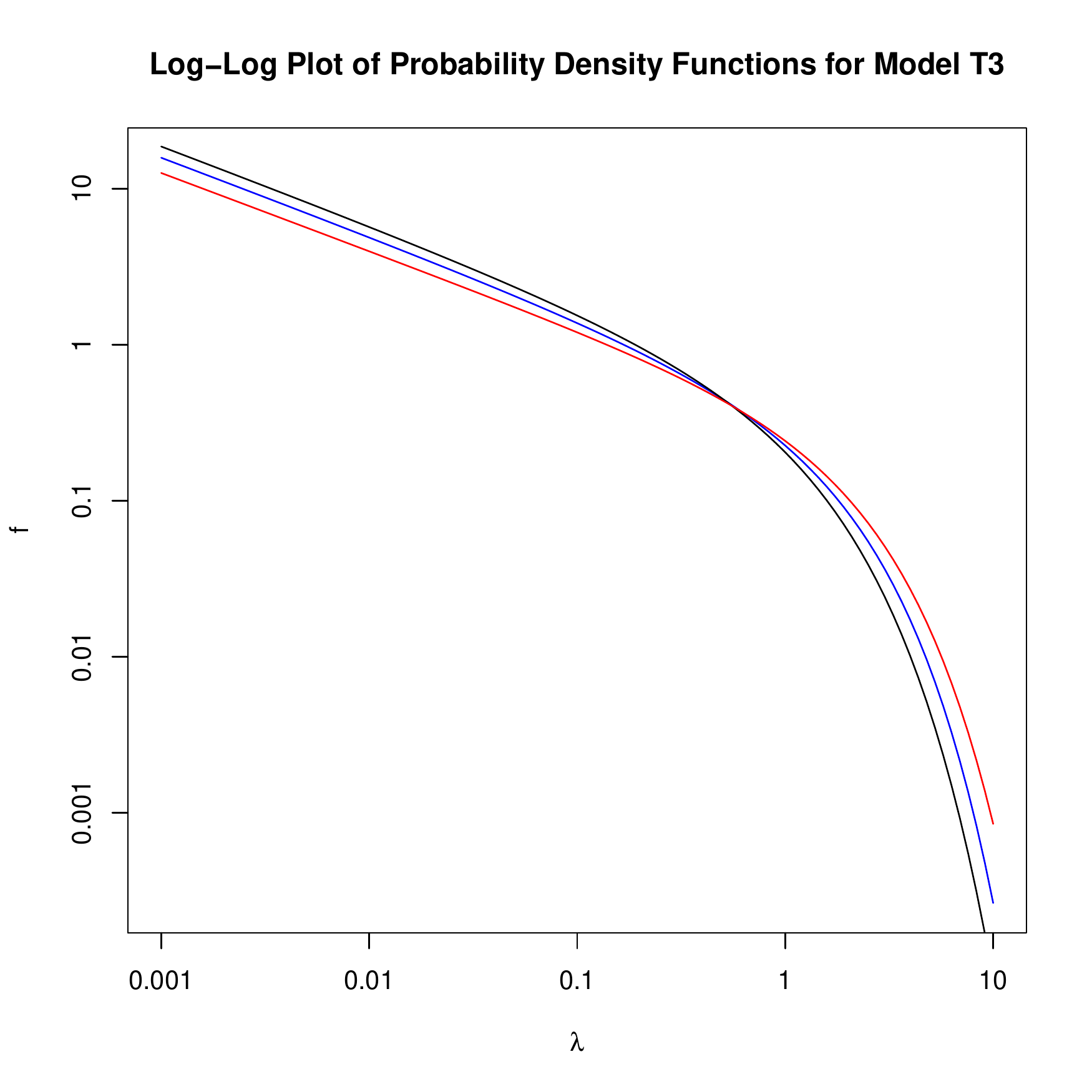}
\hspace*{\fill}
\caption{Log-log plot of three approximating density functions over part of their support $\lambda\in\left(0,\infty\right)$. 
The asymptotic density of Equation~\eqref{eq:T3pdf} at the singularity $\mu_0=0$ ($\phi_0=1$) is in black; the approximating
density at the nearby parameter value $\mu_0=1$ ($\phi_0\approx0.9993$ and $n=10^6$) is in blue; and the asymptotic density 
at non-singular points, the $\chi_1^2$ distribution, is in red.  The blue approximating density can be viewed as an 
interpolant of the two asymptotic densities
at and near the singularity.}
\label{loglogplotT3}
\end{figure}

One can show that for $\phi_0\in\left (0,1\right )$ as $n\to \infty$ Equation~\eqref{eq:T3pdf} gives the probability density function of $\chi_1^2$.

Although Proposition~\ref{prop:T3pdf} expresses the probability density function in terms of the error function, 
this density can quickly be integrated numerically to obtain a highly accurate approximation.

\medskip

Figure~\ref{loglogplotT3} compares the density functions of Equation~\eqref{eq:T3pdf} 
at the singularity $\mu_0=0$ ($\phi_0=1$) and a regular point near the singularity
$\mu_0=1$ ($\phi_0\approx{}0.9993$ when $n=10^6$) to that of $\chi_1^2$.
At the singularity, the asymptotic density is given exactly by Equation~\eqref{eq:T3pdf}, since there is no dependence on $n$. 
At all other points $\mu_0 > 0$, the asymptotic density is given by $\chi_1^2$.
The density plot for the parameter near the singularity, at $\mu_0=1$, lies between the other two plots,
and can be considered a sort of interpolant that depends both on the sample size $n$ and value of the parameter $\phi_0$. 
Unlike the asymptotic densities, which 
have a jump discontinuity 
at the singularity, the density of
Equation~\eqref{eq:T3pdf} is a continuous function of $\phi_0 \in \left[0,1\right)$
for any fixed $n$.

\subsection*{Simulations}

\begin{figure}[!htb]
\hspace*{\fill}
\begin{subfigure}{.44\textwidth}
  \includegraphics[width=1\linewidth]{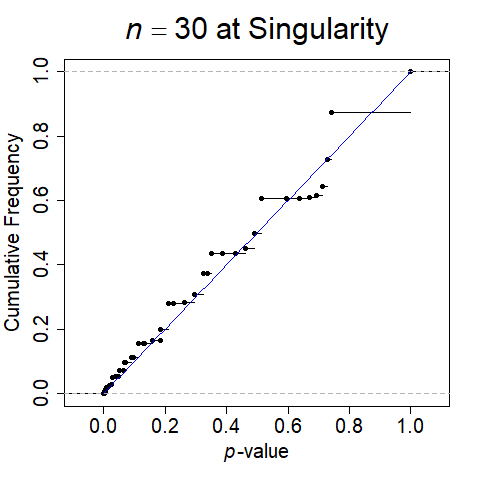}
\end{subfigure}
\hfill
\begin{subfigure}{.44\textwidth}
  \includegraphics[width=1\linewidth]{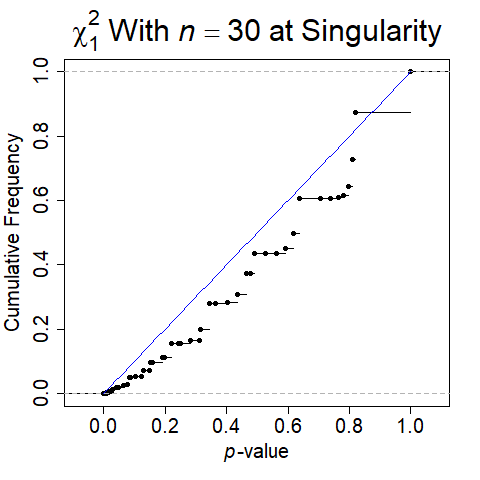}
\end{subfigure}
\hspace*{\fill} \\
\vspace*{\fill}
\hspace*{\fill}
\begin{subfigure}{.44\textwidth}
  \includegraphics[width=1\linewidth]{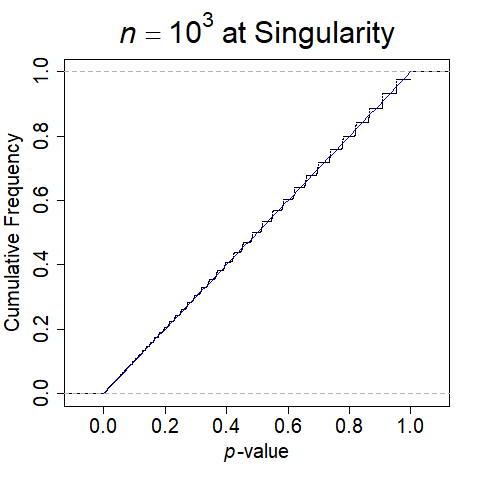}
\end{subfigure}
\hfill
\begin{subfigure}{.44\textwidth}
  \includegraphics[width=1\linewidth]{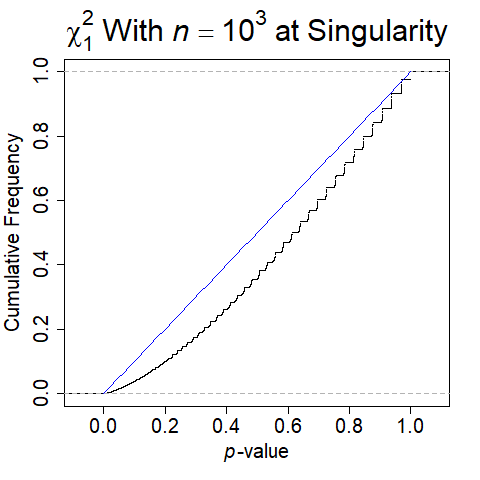}
\end{subfigure}
\hspace*{\fill} \\
\vspace*{\fill}
\hspace*{\fill}
\begin{subfigure}{.44\textwidth}
  \includegraphics[width=1\linewidth]{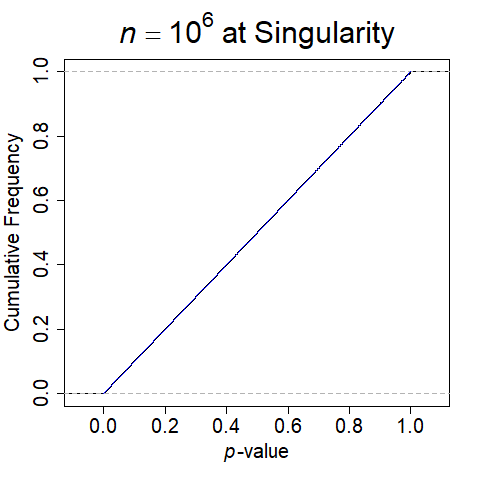}
\end{subfigure}
\hfill
\begin{subfigure}{.44\textwidth}
  \includegraphics[width=1\linewidth]{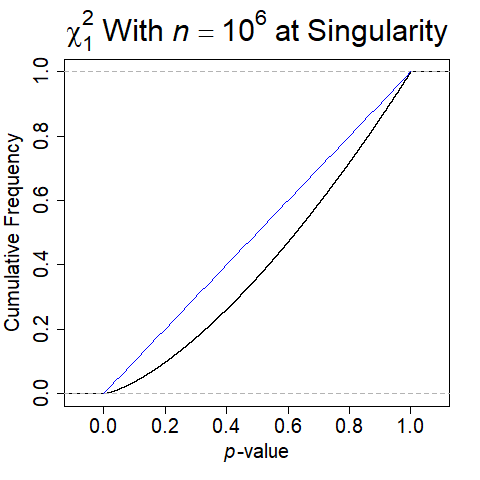}
\end{subfigure}
\hspace*{\fill} \\
\caption{Empirical cumulative distribution functions of $p$-values for the density function of  Equation~\eqref{eq:T3pdf} (left column) and the $\chi_{1}^2$ approximation (right column) for samples sizes $n=30, 10^3, 10^6$ computed at the singularity for model T3. The diagonal line, representing ideal behavior, is shown for comparison.}\label{interpdistsims3}
\end{figure}

We performed simulations to compare the use of the probability density function of Equation~\eqref{eq:T3pdf} to 
the $\chi_{1}^{2}$ density for determining $p$-values of the likelihood ratio statistic when testing $H_0$ vs. $H_1$.
We focused on true parameter values both at ($\mu_0 = 0)$ and near the singularity ($\mu_0 = 1$, $n$ varying). 
Near the singularity both distributions agree asymptotically, but at the singularity the $\chi_1^2$ distribution is not the 
asymptotic distribution, while that of Equation~\eqref{eq:T3pdf} is.  
As the $\chi_1^2$ distribution might naively be applied by an empiricist at the singularity, this last comparison is relevant.
The value $\mu_0=1$ was chosen to be near enough, but not too near, to the singularity 
so that the $\chi_{1}^{2}$ distribution and the asymptotic distribution at the singularity were both poor approximations.
A range of sample sizes was chosen, in part to demonstrate that near the singularity the $\chi_{1}^{2}$ distribution 
can perform relatively poorly even for a large sample size, despite its being the asymptotic distribution. 

Specifically, for the simulations presented in Figures \ref{interpdistsims3}, \ref{interpdistsims4}, (and later in Figures
\ref{fig:sim1tree} and \ref{fig:sim1treenb}), $\theta_0 \in \Theta_0$ was chosen making $\mu_0=0$ or $1$ 
for sample sizes $n = 30$, $10^3$, $10^6$.  For each setting, $\mu_0$, $n$, data was simulated from the 
multinomial distribution $10^{6}$ times, and likelihood ratio statistics were calculated for each replicate. 
The probability density functions of Proposition~\ref{prop:T3pdf} were used to determine $p$-values 
by numerical integration from the observed value of the statistic to infinity; $p$-values 
were also calculated using the $\chi_{1}^{2}$ approximation by standard software.  For each setting an empirical cumulative distribution function for $10^6$ $p$-values was graphed.  
 
In Figures~\ref{interpdistsims3}~and~\ref{interpdistsims4}, the discrete nature of the multinomial distribution is strongly apparent,
particularly for $n=30$.  Since the possible likelihood ratio statistics form a discrete set and are unevenly spaced, jumps in the 
cumulative plots of $p$-values are unavoidable regardless of the simulation size.

Ideally, when $\Theta_0$ has lower dimension than $\widetilde{\Theta}$ (unlike Example~\ref{ex:coin}) as for 
model T3, an approximate density function for the likelihood ratio statistic produces a simulated 
empirical cumulative distribution function of $p$-values close to $F_X\left(x\right)=x$ for $x\in\left(0,1\right)$.  
The left column of Figure~\ref{interpdistsims3} shows that this holds for the density function of 
Equation~\eqref{eq:T3pdf} for the singularity, even for a relatively small sample size of $n=30$. 
In contrast, this fails for the $\chi_1^2$ distribution (which is not the asymptotic distribution),
as seen in the right column.

\begin{figure}[!htb]
\hspace*{\fill}
\begin{subfigure}{.44\textwidth}
  \includegraphics[width=1\linewidth]{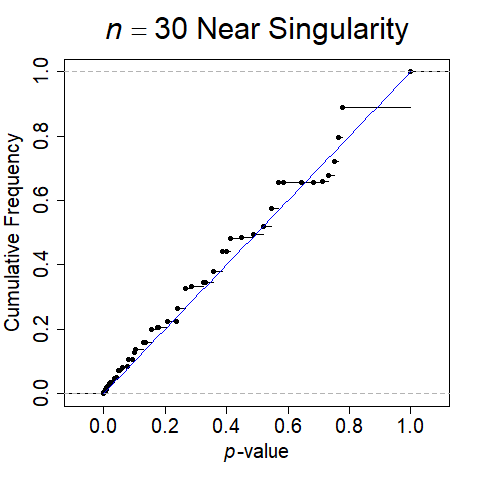}
\end{subfigure}
\hfill
\begin{subfigure}{.44\textwidth}
  \includegraphics[width=1\linewidth]{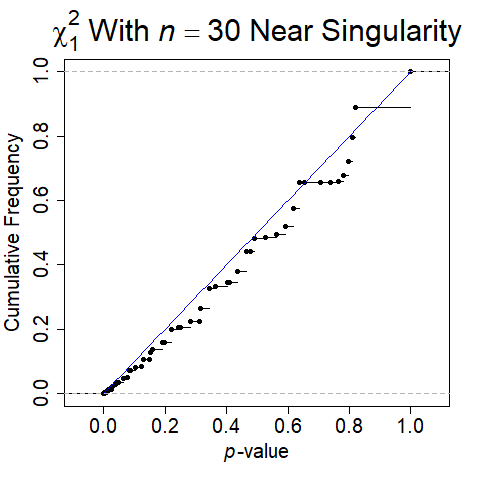}
\end{subfigure}
\hspace*{\fill} \\
\vspace*{\fill}
\hspace*{\fill}
\begin{subfigure}{.44\textwidth}
  \includegraphics[width=1\linewidth]{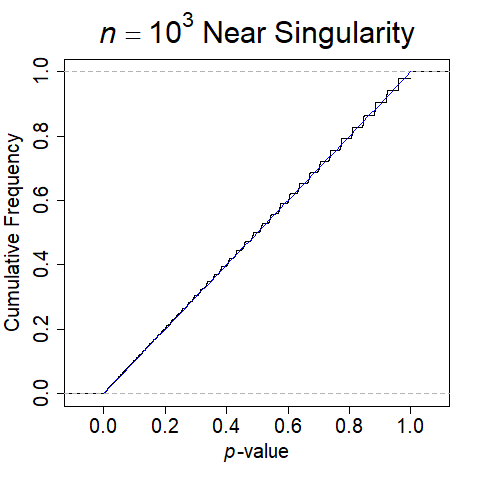}
\end{subfigure}
\hfill
\begin{subfigure}{.44\textwidth}
  \includegraphics[width=1\linewidth]{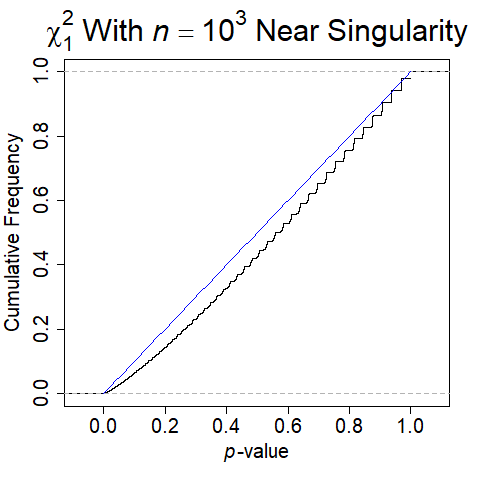}
\end{subfigure}
\hspace*{\fill} \\
\vspace*{\fill}
\hspace*{\fill}
\begin{subfigure}{.44\textwidth}
  \includegraphics[width=1\linewidth]{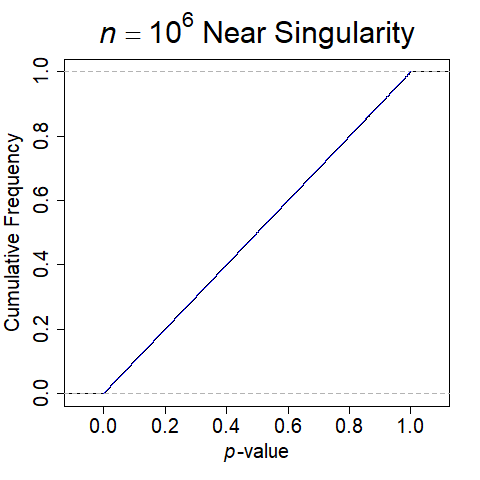}
\end{subfigure}
\hfill
\begin{subfigure}{.44\textwidth}
  \includegraphics[width=1\linewidth]{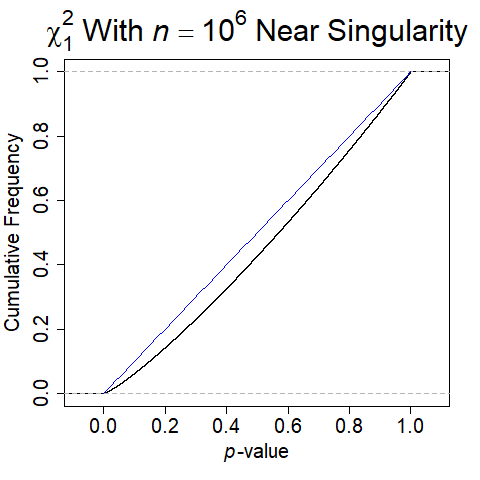}
\end{subfigure}
\hspace*{\fill} \\
\caption{Empirical cumulative distribution functions of $p$-values for the density function of Equation~\eqref{eq:T3pdf} 
(left column) and the $\chi_{1}^2$ approximation (right column) for sample sizes $n=30, 10^3, 10^6$ computed near the singularity, 
$\mu_0=1$, for model T3. The diagonal line, representing ideal behavior, is shown for comparison.}
\label{interpdistsims4}
\end{figure}

In Figure~\ref{interpdistsims4}, the results of these simulations are shown for the parameter near the singularity.
Again, plots in the left column show that the density function of Equation~\eqref{eq:T3pdf} performs extremely well, 
even for a sample size of $n=30$. The right column illustrates that the $\chi_{1}^{2}$ distribution is a poor approximation
for each of the three sample sizes, even though it is the asymptotic distribution. 
As an approximate density, the $\chi_1^2$ performs better here than at the singularity where it is not the asymptotic distribution,
but not as well as the approximating density of $\widetilde \Lambda_n$.
In summary, naively assuming the $\chi_1^2$ distribution is an accurate approximation for the
likelihood ratio statistic near (or at) a singularity can lead to inaccurate estimates of $p$-values.

Significantly, the right columns of Figures~\ref{interpdistsims3}~and~\ref{interpdistsims4} suggest that the use of the 
$\chi_{1}^{2}$ distribution gives a conservative test, as it produces larger $p$-values than desired, 
leading to rejecting $H_{0}$ less often than desired. Moreover, such a test is increasingly conservative 
closer to the singularity. This behavior has an intuitive geometric interpretation: When $\theta_0$ is on the vertical
line segment of $\Theta_0$ and near, but not at, the singularity, then 
the observation can be substantially closer to an incorrect segment of $\Theta_0$ than to the correct segment. The observation is then interpreted to be less extreme than it should be. Use of the 
$\chi_{1}^{2}$ distribution then gives a larger $p$-value than desired.

\section{Application to Model T1}\label{sec:T1}

We now examine our second example, model T1, in which the null hypothesis is that the species tree has a specific topology.

Our two hypotheses for this test are:
\begin{align*}
H_{0}\tc \quad{}&\Theta_{0}=\left\{\left(1-\frac{2}{3}\phi_{0},\frac{1}{3}\phi_{0},\frac{1}{3}\phi_{0}\right)\right\}, \text{ with $\phi_{0}=e^{-t} \in \left(0,1\right]$,} \\
H_{1}\tc \quad{}&\Theta_{1}=\Delta^2\smallsetminus\Theta_0.
\end{align*}

The model $\widetilde{\Theta}=\Theta_0\cup\Theta_1$ is again the open probability simplex $\Delta^2$, which is viewed as a subset of $\mathbb{R}^2$ through the same affine transformation used for model T3. This is as depicted in Figure~\ref{transformedSimplex}, but with the two non-vertical line segments erased.

Applying Theorem~\ref{maintheorem}, an
approximate distribution of the likelihood ratio statistic can be found. The proof of the following is given in Appendix~\ref{app2}.

\begin{prop} \label{prop:T1approx} 
For model T1, the likelihood ratio statistic for testing $H_0$ vs. $H_1$ at a true parameter point $\theta_0= \left(1-\frac{2}{3}\phi_{0},\frac{1}{3}\phi_{0},\frac{1}{3}\phi_{0}\right)$ with sample size $n$ is approximately distributed as the random variable
\begin{equation*}
\widetilde{\Lambda}_n=Z^{2}+\frac{1}{2}\left(1-\sgn\left(\bar{Z}\right)\right)\bar{Z}^{2},\end{equation*}
where $Z\sim{}\mathcal{N}\left(0,1\right)$, $\bar{Z}\sim{}\mathcal{N}\left(\mu_{0},1\right)$ and $\mu_{0}=\sqrt{2n}\frac{1-\phi_{0}}{\sqrt{\phi_{0}\left(3-2\phi_{0}\right)}}$. 
\end{prop}

Note that the distribution is the same as the first argument of the minimum in the distribution in Proposition~\ref{prop:T3approx} for model T3. 
This is expected as the first argument referred to the single line segment which is $\Theta_0$ in this example.

Again, if $\sgn\left(\bar{Z}\right)$ was always positive then the distribution would be a $\chi_{1}^{2}$ distribution, 
while if $\sgn\left(\bar{Z}\right)$ was always negative then it would be a $\chi_{2}^{2}$ distribution. 
Further calculations in  Appendix~\ref{app2} yield the following.

\begin{prop}\label{prop:T1pdf}
The probability density function of the random variable $\widetilde \Lambda_n$ given
for model T1 in Proposition~\ref{prop:T1approx} is,
for $\lambda > 0$,
\begin{equation}
f_{\widetilde{\Lambda}_n}\left(\lambda\right)=\frac{1}{4} \exp\left ({-\frac{\lambda}{2}}\right )\left[
\sqrt{\frac{2}{\pi \lambda}}\left(1+\erf\left(\frac{\mu_{0}}{\sqrt{2}}\right)\right) 
-\exp \left ( {-\frac{\mu_{0}^{2}}{2}}\right )M_{0}\left(\mu_{0}{}\sqrt{\lambda}\right)\right],\label{eq:T1pdf}
\end{equation}
where $M_{0}\left(x\right)=-\frac{2}{\pi}\int_{0}^{\frac{\pi}{2}}\exp\left(-x\cos\theta\right)d\theta$ is the modified Struve function
$11.5.5$ from \citet{olver2010nist} for real numbers $x$ and $\mu_{0}=\sqrt{2n}\frac{1-\phi_{0}}{\sqrt{\phi_{0}\left(3-2\phi_{0}\right)}}$.
\end{prop}

At the singularity, where $\mu_{0}=0$, Equation~\eqref{eq:T1pdf} gives
the probability density function of $\frac 12 \chi_{1}^{2}+\frac 12\chi_{2}^{2}$. This is as one expects from Example~1.2 of \citet{drton2009}. One can also show that for $\phi_0\in\left (0,1\right )$ as $n\to \infty$ Equation~\eqref{eq:T1pdf} gives the probability density function of $\chi_1^2$,
since $M_0\left(x\right)\to{}0$ as $x\to\infty$. 

Again the approximate probability density function can be integrated numerically quickly to obtain a 
highly accurate numerical approximation.

\medskip

Figure~\ref{loglogplotT1} gives a graphical comparison of the probability density functions of 
Equation~\eqref{eq:T1pdf} at $\mu_0=1$ ($\phi_0\approx{}0.9993$ and $n=10^6$) and at $\mu_0=0$ 
(the probability density function $\frac 12 \chi_{1}^{2}+\frac 12\chi_{2}^{2}$ at the boundary) to that of $\chi_1^2$. 
The black and red densities are the asymptotic densities at and near the boundary, respectively. 
The graph for a parameter near the boundary
($\mu_0=1$)
lies between those for the asymptotic 
distributions, interpolating them in a way dependent on both sample size $n$ and parameter 
$\phi_0$. Unlike the asymptotic distributions, which jump discontinuously at the singularity, the density of
Equation~\eqref{eq:T1pdf} is a continuous function of $\phi_0$.

Note that the $\chi_1^2$ density (red curve) is closer to the approximate density (blue curve) in 
Figure~\ref{loglogplotT1} than in Figure~\ref{loglogplotT3}, indicating it is closer to our distribution for T1 than for T3. 
This is not surprising, since the derivation of the asymptotic $\chi_1^2$ is based on replacing the model with a single vertical line, which more closely matches the geometry of  the model T1 than T3. 

\begin{figure}[!htb]
\hspace*{\fill}
\includegraphics[width=.8\linewidth]{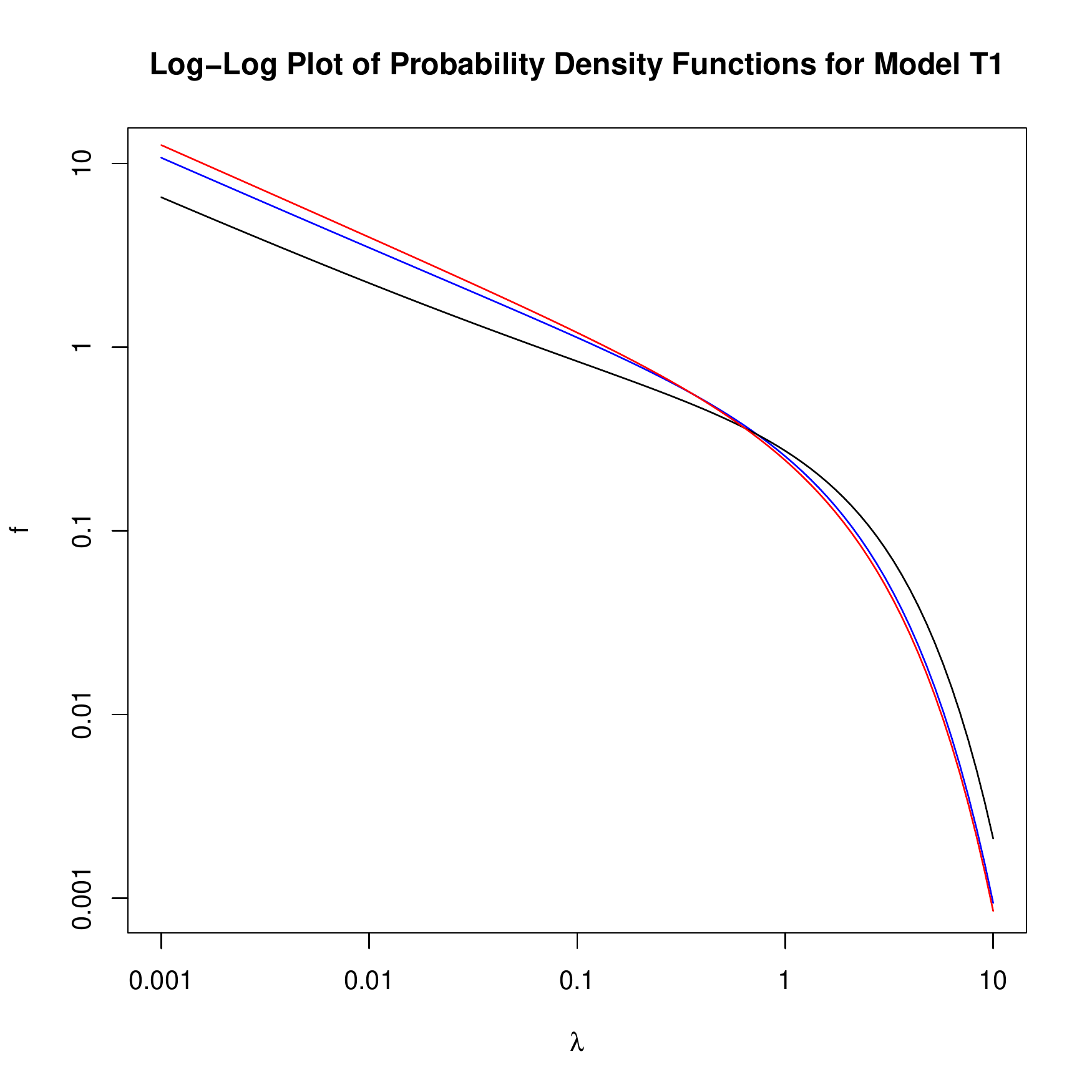}
\hspace*{\fill}
\caption{Log-log plot of three probability density functions over part of their support, $\lambda\in\left(0,\infty\right)$. The density of Equation~\eqref{eq:T1pdf} at $\mu_0=1$ ($\phi_0\approx{}0.9993$ and $n=10^6$) is in blue; the density of $\frac 12 \chi_{1}^{2}+\frac 12\chi_{2}^{2}$ of the boundary is in black; and the density of the $\chi_1^2$ distribution is in red.
The black and red plots are the asymptotic distributions at and near the boundary, respectively.}
\label{loglogplotT1}
\end{figure}

\subsection*{Simulations}

The performance of the approximate density function of Proposition~\ref{prop:T1pdf} was compared to the density function 
of the $\chi_{1}^{2}$ distribution through simulations for model T1, similar to those previously described for T3.

\begin{figure}[!htb]
\hspace*{\fill}
\begin{subfigure}{.44\textwidth}
  \includegraphics[width=1\linewidth]{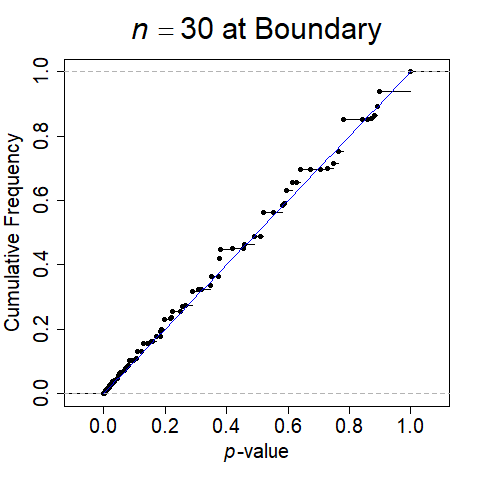}
\end{subfigure}
\hfill
\begin{subfigure}{.44\textwidth}
  \includegraphics[width=1\linewidth]{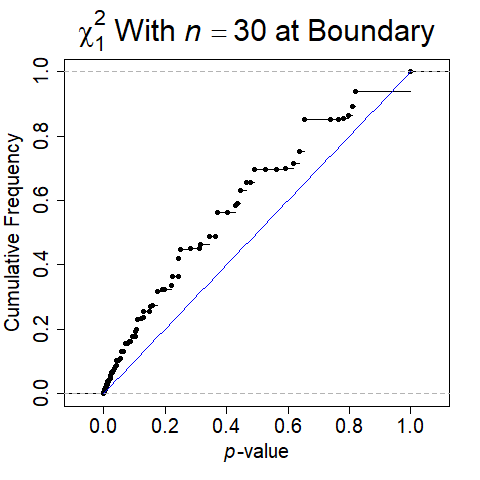}
\end{subfigure}
\hspace*{\fill} \\
\vspace*{\fill}
\hspace*{\fill}
\begin{subfigure}{.44\textwidth}
  \includegraphics[width=1\linewidth]{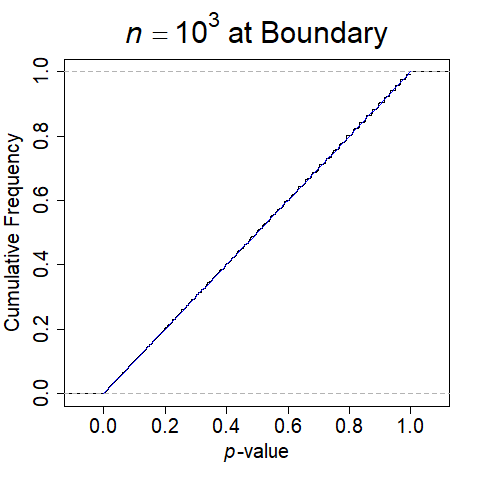}
\end{subfigure}
\hfill
\begin{subfigure}{.44\textwidth}
  \includegraphics[width=1\linewidth]{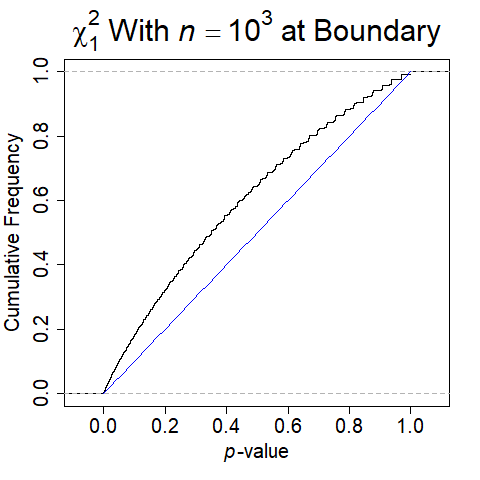}
\end{subfigure}
\hspace*{\fill} \\
\vspace*{\fill}
\hspace*{\fill}
\begin{subfigure}{.44\textwidth}
  \includegraphics[width=1\linewidth]{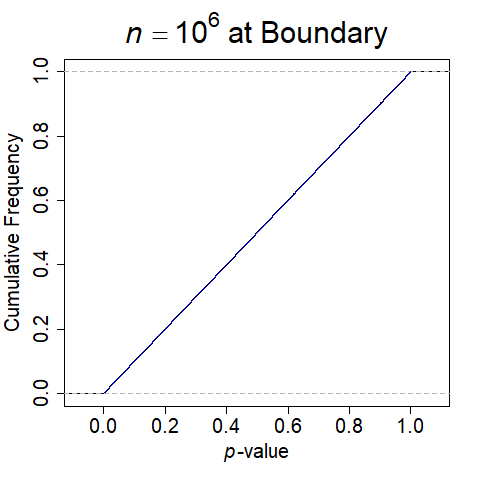}
\end{subfigure}
\hfill
\begin{subfigure}{.44\textwidth}
  \includegraphics[width=1\linewidth]{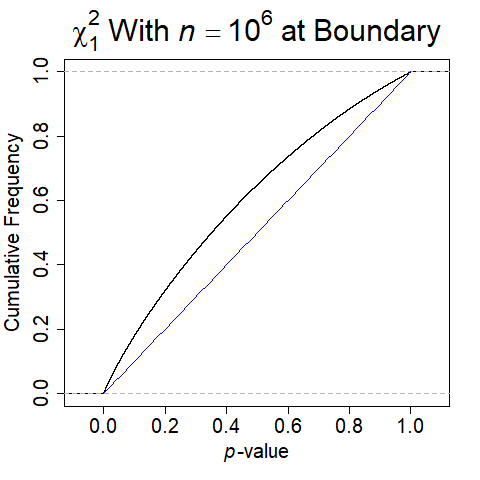}
\end{subfigure}
\hspace*{\fill} \\
\caption{Empirical cumulative distribution functions of $p$-values for the density function of 
Equation~\eqref{eq:T1pdf} (left column) and the $\chi_{1}^2$ approximation (right column)
for sample sizes $n=30, 10^3, 10^6$ computed at the boundary
for model T1. The diagonal line, representing ideal behavior, is shown for comparison.}
\label{fig:sim1tree}
\end{figure}

\begin{figure}[!htb]
\hspace*{\fill}
\begin{subfigure}{.44\textwidth}
  \includegraphics[width=1\linewidth]{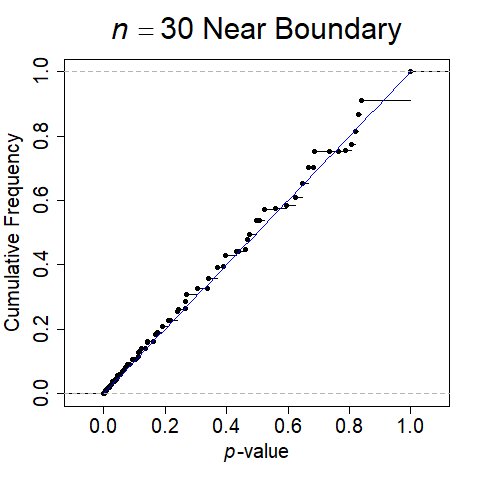}
\end{subfigure}
\hfill
\begin{subfigure}{.44\textwidth}
  \includegraphics[width=1\linewidth]{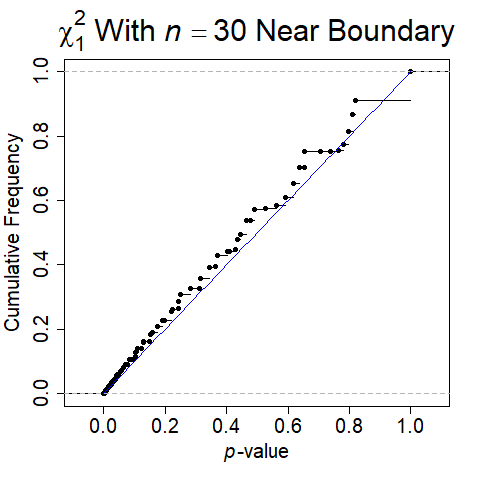}
\end{subfigure}
\hspace*{\fill} \\
\vspace*{\fill}
\hspace*{\fill}
\begin{subfigure}{.44\textwidth}
  \includegraphics[width=1\linewidth]{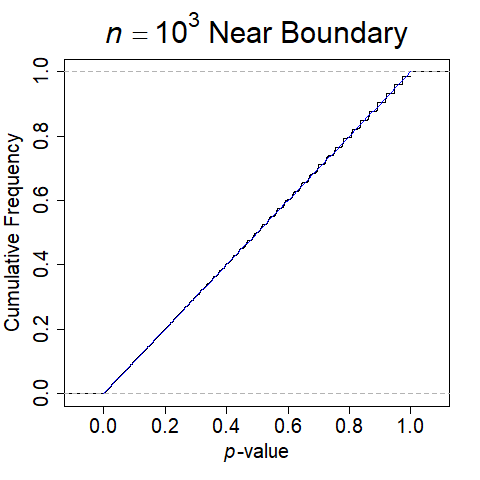}
\end{subfigure}
\hfill
\begin{subfigure}{.44\textwidth}
  \includegraphics[width=1\linewidth]{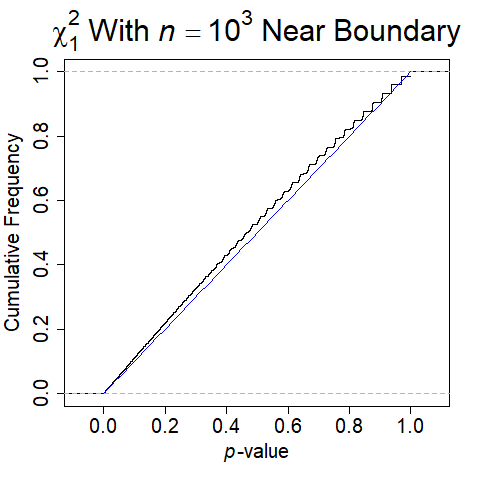}
\end{subfigure}
\hspace*{\fill} \\
\vspace*{\fill}
\hspace*{\fill}
\begin{subfigure}{.44\textwidth}
  \includegraphics[width=1\linewidth]{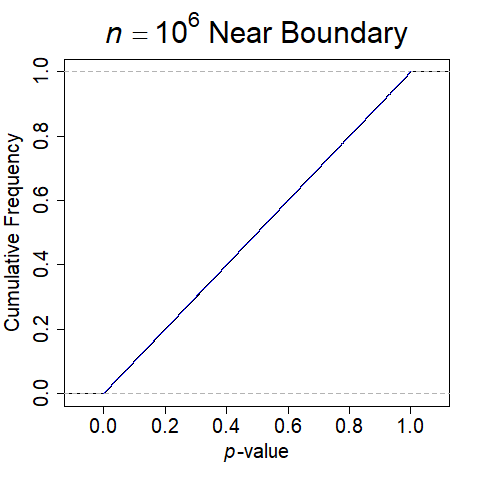}
\end{subfigure}
\hfill
\begin{subfigure}{.44\textwidth}
  \includegraphics[width=1\linewidth]{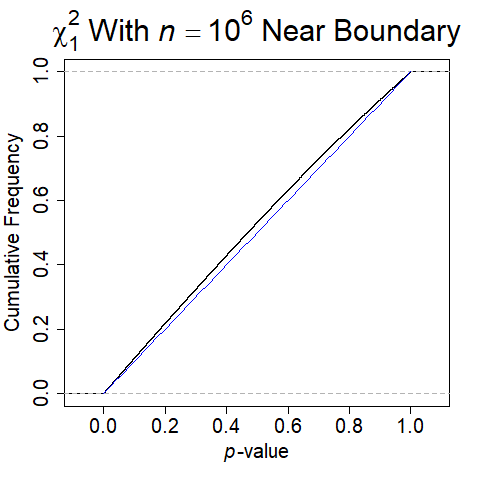}
\end{subfigure}
\hspace*{\fill} \\
\caption{Empirical cumulative distribution functions of $p$-values for the density function of 
Equation~\eqref{eq:T1pdf} and the $\chi_{1}^2$ approximation (right column) 
for samples sizes $n=30, 10^3, 10^6$ computed near the boundary,
$\mu_0=1$, for model T1. 
The diagonal line, representing ideal behavior, is shown for comparison.}
\label{fig:sim1treenb}
\end{figure}

In Figure~\ref{fig:sim1tree}, it can be seen that at the boundary our approximate density function 
outperforms the $\chi_{1}^{2}$ approximation, which is biased 
towards smaller $p$-values. This is expected, since the distribution in Proposition~\ref{prop:T1approx} 
is the asymptotic distribution and $\chi_1^2$ is not. We note that the $\chi_{1}^{2}$ approximation 
rejects $H_{0}$ more often than it should and thus gives an anti-conservative test.

Near the boundary, as shown in Figure~\ref{fig:sim1treenb}, our probability density function 
again fits the distribution of the likelihood ratio statistic better than 
the $\chi_{1}^{2}$ does, though the improvement is minimal compared to that in Figure~\ref{fig:sim1tree}
for the boundary.  This is expected as the $\chi_{1}^{2}$ is now the asymptotic distribution. Moving away from the boundary 
(simulations not shown), the $\chi_{1}^{2}$ distribution becomes a progressively better approximation, 
but remains biased towards smaller $p$-values. Thus the use of the $\chi_{1}^{2}$ approximation 
leads to rejection of $H_{0}$ more often than it should, and is anti-conservative. Again, the 
$\chi_{1}^{2}$ performs better for model T1 than for model T3 for some $\mu_0$.

\medskip

The anti-conservative behavior of the $\chi_{1}^{2}$ distribution is geometrically intuitive. For a true parameter 
$\theta_0$ near the boundary point of $\Theta_0$, some sample points will lie lower than the boundary, 
giving an MLE that is the boundary point. Such sample points are thus further from the MLE than they 
are from the vertical line extending $\Theta_0$. However, the $\chi_1^2$ distribution is appropriate 
for judging their squared distance from that line. This causes them to be viewed as more extreme than they 
should be, and their $p$-values to be calculated as smaller than desired.

\section{Approximating likelihood ratio statistic distributions with \texorpdfstring{$\chi^2$}{chi-squared}}
\label{totalvariation}

The distributions of Propositions~\ref{prop:T3approx}~and~\ref{prop:T1approx} interpolate 
between the asymptotic distribution at the singularity or boundary, respectively, and the asymptotic $\chi_1^2$ distribution 
far from the singularity or boundary.
The further the true parameter point is from the singularity or boundary, the more accurate the 
$\chi_1^2$ approximation is.

Indeed, while we have shown these approximate distributions for likelihood ratio statistics perform better than the 
asymptotic ones for finite sample sizes near the singularities and boundaries of our example models, it may 
still be desirable to use the asymptotic $\chi_1^2$ distribution for testing sufficiently far from those points. The simpler 
form of these distributions and ready availability in standard software remains attractive. A natural problem, 
then, is how to decide when the simpler distribution is likely to lead to adequate performance in testing.

To approach this question quantitatively, we employ the \emph{total variation distance} between our 
approximate distributions and the $\chi_1^2$. 
The total variation distance 
between two continuous probability distributions $F$, $G$, with densities $f$, $g$, of support $R$, is
\begin{equation*}
\delta\left(F,G\right)=\frac{1}{2}\int_{R}\mid{}f\left(\lambda\right)-{g}\left(\lambda\right)\mid{}d\lambda,
\end{equation*}
which can be interpreted as the maximum absolute difference  of probabilities of events.

Using the distribution of Proposition~\ref{prop:T3approx} or Proposition~\ref{prop:T1approx}, one can choose an 
acceptable upper bound $\epsilon$ on the total variation distance between this distribution and the $\chi_1^2$. 
Then, using a numerical optimization routine, one can determine the values of $\phi_0,n$ for which this bound is not 
exceeded. The $\chi_1^2$ approximation might be considered acceptable for such $\phi_0$ and $n$.

\subsection*{Application to Model T1}

For model T1, the dependence of the distribution 
from Proposition~\ref{prop:T1approx} 
on $\phi_0$ and $n$ is only through $\mu_0=\mu_0\left(\phi_0,n\right)$, so let $F_{\mu_0}$ denote this distribution
viewed as a function of $\mu_0$. From the 
derivation of the density in Appendix~\ref{app2}, it is clear that $\delta\left(F_{\mu_0},\chi_1^2\right)$ is a decreasing function 
of $\mu_0$. It is thus sufficient to determine numerically the value $\tilde \mu_0$ 
for which $\delta\left(F_{\tilde \mu_0},\chi_1^2\right)=\epsilon$. Then $\mu_0>\tilde \mu_0$ characterizes the parameters and 
sample sizes for which  the $\chi_1^2$ approximation might be considered acceptable.

Table~\ref{tab:tableT1} summarizes, for several choices of $\epsilon$, the threshold value $\tilde \mu_0$. It also shows for 
several choices of sample size $n$, the corresponding thresholds $\phi_0 < \tilde \phi_0$
and $t> \tilde t=-\log\left(\tilde\phi_0\right)$, 
since $\mu_0$ is a function of $n$ and $\phi_0$. 
For a given bound $\epsilon$, larger sample sizes allow for shorter internal branches 
of the tree in Figure~\ref{figurecoalescence}, while maintaining the $\chi_1^2$ distribution as a reasonable approximation
for the distribution of the likelihood ratio statistic.

\begin{table}[!htb]
\centering
\caption{For model T1, the threshold values $\tilde \mu_0$ are given for which $\mu_0>\tilde\mu_0$ 
ensures the total variation distance between the distribution of Proposition~\ref{prop:T1approx} and $\chi_1^2$ is less than 
$\epsilon$, for various $\epsilon$.  
For a fixed sample size $n$, the thresholds are also given in terms of $\tilde \phi_0$ or $\tilde t$.}
\label{tab:tableT1}
{\renewcommand{\arraystretch}{1.2}
\begin{tabular}{l|l|l|l|l|l|l}
 & \multicolumn{2}{c|}{$\epsilon=5\times{}10^{-3}$, } & \multicolumn{2}{c|}{$\epsilon=5\times{}10^{-4}$, } & \multicolumn{2}{c}{$\epsilon=5\times{}10^{-5}$, } \\
 & \multicolumn{2}{c|}{$\tilde\mu_0=1.84$} & \multicolumn{2}{c|}{$\tilde \mu_0=2.64$} & \multicolumn{2}{c}{$\tilde \mu_0=3.28$} \\
\hline
\multicolumn{1}{c|}{$n$} & \multicolumn{1}{c|}{$\tilde\phi_0$} & \multicolumn{1}{c|}{$\tilde t$} & \multicolumn{1}{c|}{$\tilde\phi_0$} & \multicolumn{1}{c|}{$\tilde t$} & \multicolumn{1}{c|}{$\tilde\phi_0$} & \multicolumn{1}{c}{$\tilde t$} \\
\hline
$30$ & $0.748$ & $0.291$ & $0.642$ & $0.443$ & $0.565$ & $0.572$ \\
\hline
$10^2$ & $0.863$ & $0.147$ & $0.802$ & $0.220$ & $0.754$ & $0.283$ \\
\hline
$10^3$ & $0.958$ & $0.0429$ & $0.939$ & $0.0626$ & $0.924$ & $0.0787$ \\
\hline
$10^4$ & $0.987$ & $0.01320$ & $0.981$ & $0.0190$ & $0.977$ & $0.0237$ \\
\hline
$10^5$ & $0.996$ & $0.00414$ & $0.994$ & $0.00594$ & $0.993$ & $0.00739$ \\
\hline
$10^6$ & $0.999$ & $0.00130$ & $0.998$ & $0.00187$ & $0.998$ & $0.00233$
\end{tabular}
}
\end{table}

\subsection*{Application to Model T3}

For model T3, the dependence of the density of Proposition~\ref{prop:T3pdf} on parameter $\phi_0$ and sample size $n$ is through both $\mu_0$ and $\alpha_0$. 
However, it is 
clear from the derivation in Appendix~\ref{app1}  that an upper bound on the variation distance 
is obtained by setting $\alpha_0$ to its minimum value, $\alpha_0=\arctan{\left(\frac{1}{3}\right)}$, 
for any value of $\mu_0$. This simplifies the computations and leads to a conservative estimate of the 
threshold $\tilde \mu_0$. Table~\ref{tab:tableT3} summarizes thresholds found in this way. 

\begin{table}[!htb]
\centering
\caption{For model T3, conservative threshold values $\tilde \mu_0$ are given for which $\mu_0>\tilde\mu_0$ ensures the  total variation distance 
between the distribution of Proposition~\ref{prop:T3approx} and $\chi_1^2$ is less than $\epsilon$,  for various $\epsilon$. 
For a fixed sample size $n$, the thresholds are also given in terms of $\tilde \phi_0$ or $\tilde t$.}
\label{tab:tableT3}
{\renewcommand{\arraystretch}{1.2}
\begin{tabular}{l|l|l|l|l|l|l}
 & \multicolumn{2}{c|}{$\epsilon=5\times{}10^{-3}$, } & \multicolumn{2}{c|}{$\epsilon=5\times{}10^{-4}$, } & \multicolumn{2}{c}{$\epsilon=5\times{}10^{-5}$, } \\
 & \multicolumn{2}{c|}{$\tilde \mu_0=2.74$} & \multicolumn{2}{c|}{$\tilde\mu_0=3.64$} & \multicolumn{2}{c}{$\tilde\mu_0=4.40$} \\
\hline
\multicolumn{1}{c|}{$n$} & \multicolumn{1}{c|}{$\tilde\phi_0$} & \multicolumn{1}{c|}{$\tilde t$} & \multicolumn{1}{c|}{$\tilde \phi_0$} & \multicolumn{1}{c|}{$\tilde t$} & \multicolumn{1}{c|}{$\tilde \phi_0$} & \multicolumn{1}{c}{$\tilde t$} \\
\hline
$30$ & $0.629$ & $0.463$ & $0.525$ & $0.645$ & $0.449$ & $0.802$ \\
\hline
$10^2$ & $0.795$ & $0.230$ & $0.727$ & $0.319$ & $0.672$ & $0.398$ \\
\hline
$10^3$ & $0.937$ & $0.0650$ & $0.916$ & $0.0879$ & $0.898$ & $0.108$ \\
\hline
$10^4$ & $0.980$ & $0.0198$ & $0.974$ & $0.0264$ & $0.968$ & $0.0321$ \\
\hline
$10^5$ & $0.994$ & $0.00617$ & $0.992$ & $0.00820$ & $0.990$ & $0.00993$ \\
\hline
$10^6$ & $0.998$ & $0.00194$ & $0.997$ & $0.00258$ & $0.997$ & $0.00312$
\end{tabular}
}
\end{table}

\section{Discussion}\label{sec:discuss}

As the examples of models T1 and T3 illustrate, not only should we expect non-standard asymptotic distributions for 
hypothesis testing at singularities and boundaries of models, but that even near such points the standard $\chi^2$ 
asymptotic distributions may behave poorly when testing. Although increasing sample size may 
lead to better performance at any specific point, the discontinuous behavior of the asymptotic distribution means a 
region of poor performance can remain, though it shrinks in size. While \citet{drton2009} commented that convergence 
to the asymptotics can be slow near a boundary or singularity, we further emphasize that the nonuniformity of the 
rate of convergence poses even more of a problem. 
Unless we have an  \emph{a priori} quantitative bound separating the true parameter from the singularities and boundaries, 
no finite sample size can be found which will lead to uniformly good performance of the standard asymptotic approximation.

Moreover, depending on the model, use of the $\chi^2$ asymptotic approximation may lead to either conservative or anti-conservative tests (or both, in different regions), depending on the geometry of the model beyond the singularity or boundary. Thus no simple rule can be adopted for adjusting one's test.
Theorem~\ref{maintheorem} suggests one alternative approach, of avoiding the approximation of the model by its tangent cone inherent in the derivation of the asymptotic distribution, and using a
different  approximate distribution dependent on both the true parameter and the sample size. For our example models this performed well, as illustrated by our simulations.

Even for our models, there are a number of hypothesis tests not presented here for which Theorem~\ref{maintheorem} will be useful. 
For instance, one may wish to test whether data fits a null hypothesis of a particular tree, model T1, vs. an alternative of the other trees, model T3$\smallsetminus$T1. Failure to reject the null hypothesis for each of the three choices of T1 would, in biological terminology, be interpreted as a soft polytomy,  where an unresolved (star) tree represents ignorance of the true resolution. Similarly, one may wish to test whether data fits a simple hypothesis of an unresolved  tree, $\theta_0=\left(\frac{1}{3},\frac{1}{3},\frac{1}{3}\right),$ vs. an alternative of a resolved tree, model $T_3\smallsetminus \{\theta_0\}$. For this test failure to reject the null hypothesis would, in biological terminology, be interpreted as a hard polytomy, where an unresolved tree represents
what are believed to be true relationships. 

\smallskip
Within phylogenetics, another  possible use of Theorem~\ref{maintheorem} is for conducting hypothesis tests for distance data to fit a tree. An evolutionary distance $d\left(a,b\right)$ is typically a numerical measure of the amount of mutation between two species $a$ and $b$, and under certain modeling assumptions should in expectation match the sum of lengths of branches between them on a tree.
The $3$-point condition states that for an ultrametric tree to exist relating species $a$, $b$, $c$, the expectations of
 $d\left(a,b\right)$, $d\left(a,c\right)$, and $d\left(b,c\right)$
 must have the two largest equal, with the smallest pair indicating the correct tree topology.
 This is similar to models T1 and T3, with the inequality reversed, except that the distances may have any non-negative values. Again the model has a singularity or boundary.
 
Several works \cite{gu1996bias,massingham2007statistics} have proposed statistical tests involving distances. For instance, \citet{gu1996bias} tested the $3$-point condition by focusing on the difference of the two distances that are assumed to be equal under $H_0$. Arguing that this difference is asymptotically normally distributed, a $Z$-test is performed. However, when all three distances are near equal, as they would be near the singularity or boundary point corresponding to a star tree, this test becomes inaccurate, as the smallest value may well not correspond to the true topology. Just as with models T1 and T3, the test could either be anti-conservative or conservative, depending on whether the null hypothesis was of a specific $3$-species ultrametric tree or of any of the three possible trees, respectively.

\medskip

Our example models have rather special structure making them amenable to our approach. Since $\Theta_0$ was locally linear, except at the singularities and boundaries, we were able to compute explicit density functions for the relevant distributions, so that using them was no more difficult than using a $\chi^2$. Our examples also had the interesting feature that in the biological application one is often most interested in data around the singularity or boundary, and so effective hypothesis testing in that region is of special concern.
Although we do not believe similar calculations of our approximate finite sample distribution will be tractable for all models, there are likely to be some where this approach will prove useful.  For models that are not amenable to such calculations, special attention still needs to be paid near singularities and boundaries, perhaps through the use of parametric bootstrapping to obtain approximations of the distribution. 

\smallskip

With a broader perspective, Theorem~\ref{maintheorem} suggests that whenever the asymptotic distribution performs badly for hypothesis testing, one might do better by using a distribution taking the local geometry of the model into account in a more subtle way than just through the tangent cone. For instance, if a model were described by a curve in the plane, one should expect that even at regular points the asymptotic distribution may be less useful in regions of high curvature, where the tangent cone approximation of the model is poor. However, unlike in the case of singularities or boundaries one should be able to work out a sample size ensuring a reasonable fit by a $\chi^2$, as long as the curvature is bounded. If obtaining a data set of that size is not possible, then 
even if the distribution of Theorem~\ref{maintheorem} cannot be computed, first approximating the model by a simpler curve with similar curvature, such as an appropriate polynomial, and then using the theorem might lead to a better distribution for hypothesis testing. Failing that, parametric bootstrapping again remains an option.

\section*{Acknowledgements}

This work is supported by the US National Institutes of Health grant R01 GM117590, awarded under the Joint DMS/NIGMS Initiative to Support Research at the Interface of the Biological and Mathematical Sciences.

\appendix

\section{The multispecies coalescent model}\label{app:MSC}

We briefly introduce the multispecies coalescent model, which underlies models T1 and T3 of Examples~\ref{ex:1tree}~and~\ref{ex:3tree}.
This model, introduced by \citet{PamiloNei1988} (see also \citep{RannalaYang2003}), extends the Kingman coalescent model of population genetics, from applying to a single population, to a tree of populations, called a species tree. It is the fundamental model of the biological process of \emph{incomplete lineage sorting}, by which gene trees of sampled lineages can fail to match the structure of the tree relating species overall. Incomplete lineage sorting is one of several processes that can make inference of species relationships from genetic data difficult. An example of a single such gene tree sampled for a particular species tree is shown in Figure~\ref{figurecoalescence}.

The Kingman coalescent models a finite number of lineages, traced backward in time within a single population, as they merge, or coalesce, at common ancestors.
The most convenient time scale is in  coalescent units $t$, where $\Delta t=\Delta \tau/N\left(\tau\right)$, with time $\tau$ measured in number of generations and $N\left(\tau\right)$ the population size. In these units, if $k$ lineages are sampled, the time to the first coalescence of the first pair of lineages
is exponentially distributed with rate $\binom k2$. The pair that coalesces is then chosen uniformly at random. Then the coalescent process begins again with one less lineage, and hence rate $\binom {k-1}2$. \citet{wakeley2009coalescent} provides a comprehensive introduction to this model.

While in population genetics, one often views the Kingman model as running until all lineages coalesce to a single one, in the multispecies coalescent that may not happen within a single population, which has a finite duration.

\medskip

The parameters of the multispecies coalescent model are a rooted metric species tree, with branch lengths given in coalescent units. The branches of the species tree should be thought of as representing unstructured populations, which stretch back in time until they merge with another population. We also consider a population ancestral to the root of the species tree, which
is considered to have infinite length, so that
lineages in it coalesce into one with probability $1$. 

Specific finite numbers of lineages are to be sampled from each species' population at the leaves of the tree. Then the Kingman coalescent model applies for the duration of that population to its parental node in the tree. At that point, there are fewer lineages if any coalescent event occurred, but we gain more lineages from the other branch of the species tree which descends from that node. The combined collection of lineages then starts a new coalescent process on the branch leading towards the root. Continuing in this way, eventually a finite number of lineages reach the root, where a final Kingman coalescent process leads to a rooted metric gene tree. Finally, ignoring branch lengths yields a sampled rooted topological gene tree.

\medskip
While for species trees with many species it is difficult to compute the probability of any gene tree (e.g., \citet{rosenberg2002probability}),
in the applications based on models T1 and T3, the species tree has only three species, and only one lineage is sampled from each. With only one lineage per species, coalescence can occur only in the internal branch of the tree or ``above-the-root'', and not in any terminal branch. Thus the only relevant branch length is the internal one.

Suppose that the true species tree is a rooted three species tree $\left(\left(a,b\right)\tc t,c\right)$, as shown in Figure~\ref{figurecoalescence}. There are three possible gene tree topologies, \begin{equation*}AB|C, \ AC|B,\ BC|A.\end{equation*} 

In this case, the probability of gene trees discordant from the species tree are easiest to compute. For instance the gene trees $AC|B$ and $BC|A$ can only form if no coalescence occurs except above the root. From the exponential distribution of coalescent times, the probability of no coalescence of two lineages in a branch of length $t$ is $e^{-t}$. Then, with three lineages present at the root, due to the exchangeability of lineages, the formation of each of the three rooted trees must have equal probability of $\frac 13$. Thus
$p_{AC|B}=\frac 13 e^{-t}$. The same argument gives
$p_{BC|A}=\frac 13 e^{-t}$, which thus implies
$p_{AB|C}=1-\frac 23 e^{-t}$.

\section{Model T3}
\label{app1}

Here we prove Proposition~\ref{prop:T3approx} and Proposition~\ref{prop:T3pdf} concerning the model T3.

\smallskip

For a fixed sample size, multinomial distributions form a regular exponential family if $\widetilde{\Theta}=\Delta$ is the open simplex. The regularity conditions of \citet{drton2009} are then satisfied, and thus Theorem~\ref{maintheorem} applies.

Since $\widetilde{\Theta}=\Delta^2$ lies on a plane in $\mathbb{R}^{3}$, we first apply an affine transformation $M:\mathbb{R}^{3}\rightarrow{}\mathbb{R}^{2}$, 
\begin{equation*}M=\begin{pmatrix}
0 & -\frac{1}{\sqrt{2}} & \frac{1}{\sqrt{2}}\\ 
\sqrt{\frac{2}{3}} & -\frac{1}{\sqrt{6}} & -\frac{1}{\sqrt{6}}
\end{pmatrix},\end{equation*}
to map $\widetilde{\Theta}$ isometrically to the plane, sending the singularity to the origin.  This maps a true parameter point, say $\theta_0 = \theta_0^{\left(1\right)} = \left(1-\frac{2}{3}\phi_{0}, \:\frac{1}{3}\phi_{0}, \:\frac{1}{3}\phi_{0}\right)$
without loss of generality, to $\left(0, \:\sqrt{\frac{2}{3}}\left(1-\phi_{0}\right)\right)$.  
Computing the Fisher information matrix $\mathcal I \left(\theta_0\right)$ for a sample of size $n = 1$ for $\theta_0$ in planar coordinates, we obtain the transformation matrix
\begin{equation*}\sqrt{n}\mathcal I\left(\theta_{0}\right)^{\frac{1}{2}}=\begin{pmatrix}
\sqrt{\frac {3n}{ \phi_{0} }}&0\\ 
0& \sqrt{\frac {3n}{ \phi_{0}  \left(  3-2\phi_{0}\right) }}
\end{pmatrix},\end{equation*}
which we apply to the planar image of $\Delta^2$.
The point $\theta_0$ is mapped to $\left(0,\mu_0\right)$ with
\begin{equation*}
\mu_{0}=\sqrt{2n}\frac{1-\phi_{0}}{\sqrt{\phi_{0}\left(3-2\phi_{0}\right)}}.
\end{equation*}

Under these transformations the null parameter space $\Theta_0$ is mapped non-conformally, provided
$\phi_0\in\left(0,1\right)$, to three line segments emanating
from the origin, one to $\left(0,\sqrt{\frac{2n}{\phi_0\left(3-2\phi_0\right)}} \right)$ passing through the true parameter point
$\left(0,\mu_0\right)$, and others to $\left(\pm \sqrt{\frac{3n}{2\phi_0}},-\sqrt{\frac{n}{2\phi_0 \left(3-2\phi_0\right)}}\right)$. 
(The parameter value $\phi_0=1$ corresponds to the singularity in $\Theta_0$ and the transformation is conformal
in this instance.)  The full parameter space $\Delta^2$ is mapped to the interior of the convex hull of the three
points given above.  See Figure~\ref{transformedSimplex}.

The angle $\alpha_0 > 0$ formed between the positive $x$-axis and 
the line segment joining the origin to $\left(\sqrt{\frac{3n}{2\phi_0}},-\sqrt{\frac{n}{2\phi_0 \left(3-2\phi_0\right)}}\right)$, as shown is Figure~\ref{transformedSimplex}, is
$\alpha_0 = \arctan \left(\frac{1}{\sqrt{3\left(3-2\phi_0\right)}}\right)$, and varies from $\frac{\pi}{6}$ for $\phi_0=1$ down to $\arctan\left(\frac 1 3\right)$ as $\phi_0 \to 0$.  Letting $\gamma_0 = \tan\left(\alpha_0\right)$, in the transformed space the image of $\Theta_0$ is contained in the union of the half-lines $\left\{\left(0,y\right) \mid y \ge 0\right\}$ and $y=-\gamma_0\sgn\left(x\right)x$.

By Theorem~\ref{maintheorem}, the approximate distribution of the likelihood ratio statistic is the distribution of the minimum squared Euclidean distance between a normal sample, $\mathcal{N}\left(\left(0,\mu_{0}\right),I\right)$, and three line segments in the transformed space. 
Assuming that $\theta_0$ is not too close to the boundary of $\widetilde{\Theta}_0$ in a sense dependent on sample size, little of the mass of $\mathcal{N}\left(\left(0,\mu_{0}\right),I\right)$ is outside the image of the simplex.
Thus, for the remainder of the argument, we replace these line segments with half-lines emanating from the singularity $\left(0,0\right)$.

Denote the marginal probability distributions of the bivariate normal sample by $Z\sim{}\mathcal{N}\left(0,1\right)$ and $\bar Z\sim{}\mathcal{N}\left(\mu_{0},1\right)$.
We next determine the minimum squared distance of a sample point $\left(Z, \bar Z\right)$ 
to the three half-lines.

Consider first the half-line $\left\{\left(0,y\right) \mid y \geq 0\right\}$.  If $\bar Z$ is non-negative, then the squared Euclidean distance is $Z^{2}$, while if
$\bar Z$ is negative, then the squared distance is  $Z^{2}+\bar{Z}^{2}$. 
Thus the squared Euclidean distance is 
\begin{equation}
Z^{2}+\frac{1}{2}\left(1-\sgn\left(\bar{Z}\right)\right)\bar{Z}^{2}.\label{eq:firstArg}
\end{equation}

Now consider the half-lines $y=-\gamma_0\sgn\left(x\right)x$ and denote the closest point to $\left(Z, \bar Z\right)$
by $\left( X,-\gamma_0\sgn{\left(X\right)}X\right)$. Assuming $X\ne 0$, then
$\sgn\left(X\right)=\sgn\left(Z\right)$, and minimizing 
\begin{equation*}
\left(Z-X\right)^{2}+\left(\bar{Z}+\gamma_0\sgn\left(Z\right)X\right)^{2}
\end{equation*}
yields
\begin{equation*}
X=\frac{1}{1+\gamma_0^2}\left(Z-\gamma_0\sgn\left(Z\right)\bar{Z}\right).
\end{equation*}
Substituting into the previous expression gives the squared distance
\begin{equation}
\label{eq:sqdist2}
\frac{\gamma_0^2}{1+\gamma_0^2}\left (Z+\frac 1{\gamma_0}\sgn\left(Z\right)\bar Z\right )^2=\sin^2\alpha_0\left (Z+\cot\alpha_0\sgn\left(Z\right)\bar Z\right )^2.
\end{equation}

In the case $X=0$, 
the closest point to the two half lines is the origin. This can occur only when $\bar Z \ge 0$, so the squared distance to the two half-lines is at least $Z^2$, which is 
the squared distance to the vertical half-line given in Equation~\eqref{eq:firstArg}.  Moreover,  it can be shown that $Z^2$ is at most the value given 
in Equation~\eqref{eq:sqdist2} in this case.

It follows that the approximate distribution of the likelihood ratio statistic is that of the random variable
\begin{equation*}
\widetilde{\Lambda}_n = \min\left(Z^{2}+\frac{1}{2}\left(1-\sgn\left(\bar{Z}\right)\right)\bar{Z}^{2}, ~ \sin^2\alpha_0\left (Z+\cot\alpha_0\sgn\left(Z\right)\bar Z\right )^2\right),
\end{equation*}
as given in Proposition~\ref{prop:T3approx}.

\medskip

To determine the probability density function for the approximate distribution of Proposition~\ref{prop:T3approx}, we let
$G_X\left(x\right)$ denote the cumulative distribution function of the (non-squared) Euclidean distance.  This 
can be found by integrating the bivariate normal distribution $\mathcal{N}\left(\left(0,\mu_{0}\right),I\right)$ over the tube of points within distance $x$ from the
transformed $\Theta_0$, as shown in Figure~\ref{figure4}. 
Calculations are simplified by the fact that the tubular region in Figure~\ref{figure4} has bilateral symmetry, as does the normal distribution. 

Due to symmetry, we need only integrate over the shaded regions in the figure. Let $L_i = L_i\left(x\right)$ denote the half-lines forming the outer boundaries of these regions.
Denote by $\beta_0$ the angle formed by the line segment joining the origin to the point of intersection of $L_1$ and $L_2$.
This angle has measure $\beta_0 = \frac{1}{2}\left(\frac{\pi}{2}-\alpha_0\right)$.

With this setup,
$G_X\left(x\right)= 2 \left(G_1\left(x\right)+G_2\left(x\right)+G_3\left(x\right)\right)$, where
$G_i$ is the integral over the shaded strip $R_i$, and the density of the Euclidean distance is
\begin{equation*}g_X\left(x\right)=2\left(\frac d{dx} G_1\left(x\right) +\frac d{dx} G_2\left(x\right)+\frac d{dx} G_3\left(x\right)\right ).\end{equation*}

\begin{figure}[!htb]
	\centering
		\begin{tikzpicture}[xscale=1]
        \centering
				\draw[-] (2.51,0) -- (2.51,-2.67);
				\draw[-] (2.51,-2.67) -- (0,-3.58);
				\draw[-] (2.51,-2.67) -- (5.02,-3.58);
				\draw[-,dashed] (0,-4.64) -- (2.51,-3.73);
				\draw[-,dashed] (2.51,-3.73) -- (5.02,-4.64);
				\draw[-,dashed] (1.51,0) -- (1.51,-1.97);
				\draw[-,dashed] (0,-2.52) -- (1.51,-1.97);
				\draw[-,dashed] (3.51,0) -- (3.51,-1.97);
				\draw[-,dashed] (3.51,-1.97) -- (5.02,-2.52);
				\draw[<->] (2.51,-1) -- (3.51,-1);
				\draw[-,dotted] (2.51,-2.67) -- (4.33,-1.40);
				\draw[-,dotted] (2.51,-2.67) -- (6.02,-2.67);
				\node at (3.01,-0.25) {$R_{1}$};
				\node at (4.52,-2.87) {$R_{2}$};
				\node at (4.52,-3.93) {$R_{3}$};
				\tkzDefPoint(6.02,-2.67){C}
				\tkzDefPoint(2.51,-2.67){B}
				\tkzDefPoint(5.02,-3.58){A}
				\tkzDefPoint(3.51,-1.97){D}
				\tkzDefPoint(2.51,0){E}
				\tkzMarkAngle[arrows=<->](A,B,C)
				\tkzMarkAngle[arrows=<->](C,B,D)
				\tkzLabelAngle[pos=0.6](A,B,C){$\alpha_0$}
				\tkzLabelAngle[pos=0.6](C,B,D){$\beta_0$}
				\tkzMarkRightAngle(E,B,C)
				\fill[fill=gray!50,nearly transparent]  (2.51,-2.67) -- (3.51,-1.97) -- (3.51,0) -- (2.51,0) -- cycle;
				\fill[fill=gray!100,nearly transparent]  (5.02,-2.52) -- (3.51,-1.97) -- (2.51,-2.67) -- (5.02,-3.58) --  cycle;
				\fill[fill=gray!150,nearly transparent]  (5.02,-3.58) -- (2.51,-2.67) -- (2.51,-3.73) -- (5.02,-4.64) -- cycle;
				\draw (2.51,-1.5) node[circle,fill,inner sep=1pt,label=left:{$\left(0,\mu_{0}\right)$}]{};
				\draw (3.01,-1) node[label=above:{$x$}]{};
				\draw (3.51,-0.99) node[label=right:{$L_1$}]{};
				\draw (4.27,-2.25) node[label=above:{$L_2$}]{};
				\draw (3.77,-4.19) node[label=below:{$L_3$}]{};
				\end{tikzpicture}
        \caption{The region of integration for $G_X\left(x\right)$ is between the dashed lines. The integral is evaluated as three integrals, over each of the shaded regions $R_i$.}
		\label{figure4}
\end{figure}
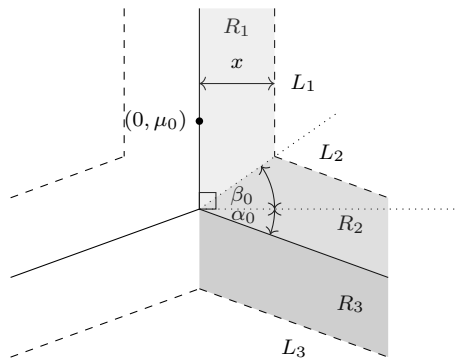

Considering $\frac d{dx} G_1 \left(x\right)$ first, one sees that this derivative is the integral of the normal density over boundary $L_1$. We show this formally using polar coordinates:
\begin{align*} 
\frac {d}{dx} G_{1}\left(x\right)
=&\int_{\beta_0}^{\frac{\pi}{2}}\frac{d}{dx}\int_{0}^{\frac{x}{\cos \beta}}\frac{1}{2\pi}\exp\left ({-\frac{1}{2}\left(r^{2}-2\mu_{0}r\sin\beta+\mu_{0}^{2}\right)}\right)r \,dr\,d\beta \\
=&\int_{\beta_0}^{\frac{\pi}{2}}
\frac{1}{2\pi}\exp\left ({-\frac{1}{2}\left({\frac{x^2}{\cos^2 \beta}}-2\mu_{0}{\frac{x}{\cos \beta}}\sin\beta+\mu_{0}^{2}\right)}\right) {\frac{x}{\cos^2 \beta}}\,d\beta\\
=&\frac{1}{2\pi}\exp\left (-\frac{1}{2} {x^2}\right ) \\
&\quad\int_{\beta_0}^{\frac{\pi}{2}}
\exp\left (-\frac{1}{2}\left ( x^2 \tan^2 \beta-2\mu_{0}x\tan \beta +\mu_{0}^{2}\right)\right) {\frac{x}{\cos^2 \beta}}\,d\beta.
\end{align*}
Substituting $y=x\tan\beta$ gives
\begin{align*}
\frac {d}{dx} G_{1}\left(x\right)
&=\frac{1}{2\pi}\exp\left (-\frac{1}{2} {x^2}\right )
\int_{x\tan\beta_0}^{\infty}
\exp\left (-\frac{1}{2}\left ( y-\mu_{0}\right)^2\right) dy
\\
&=\frac{1}{2\sqrt{2\pi}}\exp \left ({-\frac{1}{2}x^{2}}\right )\left(1-\erf\left(\frac{1}{\sqrt{2}}\left ( x\tan  \beta_0 -\mu_{0}\right)\right)\right).
\end{align*}

\smallskip

More briefly, over $R_2$ we have
\begin{equation*}
\frac d{dx}G_2\left(x\right)=\int_{L_2} f \left(z, \bar z\right) dx,
\end{equation*}
where $f$ is the density of the bivariate normal.
To evaluate this, we reflect about the line $y = \tan\beta_0 \, x$, mapping the mean of the Gaussian to
$\left(\mu_0\cos \alpha_0,-\mu_0 \sin \alpha_0\right)$, and sending $L_2$ to a vertical half-line $\left(x,y\right)$, with $y\ge x\tan \beta_0$, so
\begin{align*}
\frac d{dx}G_2\left(x\right)=&\int_{x\tan\beta_0}^\infty
\frac 1{2\pi}\exp\left (-\frac 12 \left (\left(x-\mu_0\cos\alpha_0\right)^2+\left(y+\mu_0\sin\alpha_0\right)^2\right )\right ) dy
\\
=&
\frac{1}{2\sqrt{2\pi}}\exp\left ( {-\frac{1}{2}\left(x-\mu_{0}\cos\alpha_0\right)^{2}}\right ) \\
&\quad\left( 1-\erf\left(\frac{1}{\sqrt{2}}\left(x\tan\beta_0 +\mu_{0}\sin\alpha_0\right)\right)\right).
\end{align*}
Finally, since the same reflection maps $L_3$ to the vertical half-line $\left(-x,y\right)$ with $y\ge x\tan \alpha_0$,
\begin{align*}
\frac d{dx}G_3\left(x\right)
=&
\int_{x\tan\alpha_0}^\infty \frac 1{2\pi}\exp\left (-\frac 12\left ( \left(-x-\mu_0\cos\alpha_0\right)^2+ \left(y+\mu_0\sin\alpha_0\right)^2 \right )\right ) dy\\
=&\frac{1}{2\sqrt{2\pi}}\exp \left ( {-\frac{1}{2}\left(x+\mu_{0}\cos\alpha_0\right)^{2}}\right ) \\
&\quad\left(1-\erf\left(\frac{1}{\sqrt{2}}\left(x\tan\alpha_0+\mu_{0}\sin\alpha_0\right)\right)\right).
\end{align*}

Summing these three expressions and multiplying by $2$, we obtain the density $g_{X}\left(x\right)$ for the distance. 
After a change of variable to convert to the squared Euclidean distance, the random variable $\tilde \Lambda_n$ has
density function 
\begin{align*}
&f_{\widetilde{\Lambda}_n}\left(\lambda\right)=
\frac{1}{2\sqrt{2\pi{}\lambda}}
\biggl [ 
\exp \left ({-\frac{1}{2}\lambda}\right )\left(1-\erf\left(\frac{1}{\sqrt{2}}\left ( \sqrt \lambda\tan  \beta_0 -\mu_{0}\right)\right)\right)
\\&+
\exp \left ({-\frac{1}{2}\left(\sqrt \lambda-\mu_{0}\cos\alpha_0\right)^{2}}\right )\left(1-\erf\left(\frac{1}{\sqrt{2}}\left(\sqrt \lambda\tan\beta_0 +\mu_{0}\sin\alpha_0\right)\right)\right)
\\&+
\exp\left ( {-\frac{1}{2}\left(\sqrt \lambda+\mu_{0}\cos\alpha_0\right)^{2}}\right )\left(1-\erf\left(\frac{1}{\sqrt{2}}\left(\sqrt \lambda\tan\alpha_0+\mu_{0}\sin\alpha_0\right)\right)\right)
\biggr ],
\end{align*}
as given in Proposition~\ref{prop:T3pdf}.

\section{Model T1}
\label{app2}

We now prove Propositions~\ref{prop:T1approx}~and~\ref{prop:T1pdf}, using the transformation and notation of Appendix~\ref{app1}.   
Proposition~\ref{prop:T1approx} follows immediately by a 
simple modification to the argument in Appendix~\ref{app1}.  See Equation~\eqref{eq:firstArg}.

For Proposition~\ref{prop:T1pdf}, let  $g_X\left(x\right)$ denote the probability density function for the (non-squared) distance $x$ between
a sample point $\left(Z,\bar Z\right)$ and the mean $\left(0,\mu_0\right)$.  Then $g_X\left(x\right) = \frac d{dx} G_X\left(x\right)$ is given by the integral of the Gaussian over
the dashed curves shown in Figure~\ref{figure6}.  To compute this, we integrate over the dashed boundaries of $R_1$ and $R_2$ depicted 
in the figure.  By symmetry, 
\begin{equation*}g_X\left(x\right)=\frac d{dx} G_X\left(x\right)=2\frac d{dx} G_1\left(x\right) +\frac d{dx} G_2\left(x\right),\end{equation*}
where $G_i$ is the contribution to the cdf over region $R_i$.

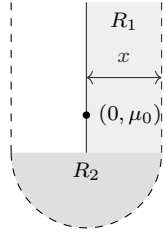
\begin{figure}[!htb]
	\centering
		\begin{tikzpicture}[xscale=1]
        \centering
				\draw[-] (1,0) -- (1,-2);
				\draw[-,dashed] (0,0) -- (0,-2);
				\draw[-,dashed] (2,0) -- (2,-2);
				\draw[<->] (1,-1) -- (2,-1);
				\draw (1.5,-1) node[label=above:{$x$}]{};
				\fill[fill=gray!50,nearly transparent]  (1,0) -- (1,-2) -- (2,-2) -- (2,0) -- cycle;
				\fill[fill=gray!100,nearly transparent] (1,-2) -- (0mm,-20mm) arc (180:360:1cm);
				\draw[-,dashed] (0,-2) arc (180:360:1cm and 1cm);
				\draw (1,-1.5) node[circle,fill,inner sep=1pt,label=right:{$\left(0,\mu_{0}\right)$}]{};
				\node at (1.5,-0.25) {$R_{1}$};
				\node at (1,-2.25) {$R_{2}$};
				\end{tikzpicture}
        \caption{The region of integration for model T1.}
		\label{figure6}
\end{figure}

For $R_1$, 
\begin{align*}
\frac d{dx} G_1\left(x\right) &= \int_0^\infty \frac 1{2\pi} \exp \left(-\frac 1 2 \left(x^2 + {(y-\mu_0)}^2 \right) \right) dy\\
&=\frac {1}{2\pi} x\exp\left (-\frac 12 \left (x^2+\mu_0^2\right )\right )\int_\pi^{2\pi}
\exp\left (\mu_{0}x\sin\theta\right ) d\theta \\
\end{align*}

On $R_2$, using polar coordinates and $C_2$ for the dashed semi-circle, we find
\begin{align*}
\frac d{dx} G_2\left(x\right) &= \int_{C_2} \frac 1{2\pi} \exp \left(-\frac 1 2 \left(x^2 + {(y-\mu_0)}^2 \right) \right) d\sigma\\
&=\frac {1}{2\pi} x\exp\left (-\frac 12 \left (x^2+\mu_0^2\right )\right )\int_\pi^{2\pi}
\exp\left (\mu_{0}x\sin\theta\right ) d\theta \\
&=-\frac {1}{2} x\exp\left (-\frac 12 \left (x^2+\mu_0^2\right )\right )M_0\left(\mu_0x\right),
\end{align*}
where the last line is obtained after a change of variables, and $M_0\left(z\right)$ is the modified Struve function $11.5.5$ from
 \citet{olver2010nist}.

\medskip

Summing, and making a change of variable $x = \sqrt{\lambda}$, we find the probability density function for $\widetilde \Lambda_n$ is
\begin{align*}
f_{\widetilde{\Lambda}_n}\left(\lambda\right)=\frac{1}{4} \exp\left ({-\frac{\lambda}{2}}\right )\left[
\sqrt{\frac{2}{\pi \lambda}}\left(1+\erf\left(\frac{\mu_{0}}{\sqrt{2}}\right)\right)-\exp \left ( {-\frac{\mu_{0}^{2}}{2}}\right )M_{0}\left(\mu_{0}{}\sqrt{\lambda}\right)\right],
\end{align*}
as given in Proposition~\ref{prop:T1pdf}.

\bibliographystyle{imsart-nameyear}
\bibliography{bibliography}

\end{document}